\documentclass[12pt]{amsart}

\usepackage {crimson}

  \usepackage{amsfonts,graphics,amsmath,amscd, amssymb,amsmath,latexsym,bbm,bm,mathtools,mathrsfs}
  \usepackage{flafter}
  \usepackage{verbatim,lipsum}
  \usepackage{enumerate}
  \usepackage{cases}
  
  \usepackage[english]{babel}
  \usepackage[utf8]{inputenc}
\usepackage[top=30mm,bottom=30mm,left=30mm,right=30mm]{geometry}
  \usepackage{tikz}
  \usetikzlibrary{cd,matrix,arrows,decorations.pathmorphing}

\definecolor{cite}{RGB}{44,123,182}
\definecolor{ref}{RGB}{215,25,28}

  \usepackage{url}
  \usepackage[pdftex,
              pdfauthor={Franco Rota},
              pdftitle={The stability manifold of local orbifold elliptic quotients},
              pdfsubject={Algebraic geometry, stability condition, elliptic root system},
              pdfkeywords={stability condition, root system, orbifold, elliptic, franco, rota},
              colorlinks=true,
              citecolor=cite,
              linkcolor=ref]{hyperref}

\newtheorem{thm}{Theorem}[section]
\newtheorem{corollary}[thm]{Corollary}
\newtheorem{lemma}[thm]{Lemma}

\newtheorem{proposition}[thm]{Proposition}

\theoremstyle{definition}
\newtheorem{definition}[thm]{Definition}
\newtheorem{example}[thm]{Example}

\newtheorem{notation}[thm]{Notation}

\newtheorem{conjecture}[thm]{Conjecture}

\newtheorem*{claim*}{Claim}
\newtheorem{remark}[thm]{Remark}

\newcommand{\C}{\mathbb{C}}
\newcommand{\Z}{\mathbb{Z}}
\newcommand{\Q}{\mathbb{Q}}
\newcommand{\R}{\mathbb{R}}

\newcommand{\HH}{\mathbb{H}}
\newcommand{\pr}[1]{\mathbb P^{#1}}

\newcommand{\oo}[1]{\omega_{#1}}

\newcommand{\HOM}[3]{\text{Hom}_{{#3}}\left( #1,#2\right) }

\DeclareMathOperator{\id}{id}
\DeclareMathOperator{\rk}{rk}

\DeclareMathOperator{\GL}{GL}

\DeclareMathOperator{\Hom}{Hom}

\DeclareMathOperator{\Ext}{Ext}
\DeclareMathOperator{\ext}{ext}

\DeclareMathOperator{\Aut}{Aut}
\DeclareMathOperator{\rad}{rad}

\DeclareMathOperator{\re}{Re}
\DeclareMathOperator{\im}{Im}

\DeclareMathOperator{\Tot}{Tot}

\DeclareMathOperator{\Pic}{Pic\,}

\DeclareMathOperator{\Hilb}{Hilb}

\DeclareMathOperator{\Stab}{Stab}
\DeclareMathOperator{\Br}{Br}

\DeclareMathOperator{\GGL}{\widetilde {\GL}^+(2, \R)}  % group GL(2,R)+

\DeclareMathOperator{\textch}{ch}												%chern classes
\DeclareMathOperator{\cch}{Ch} 

\newcommand{\ch}[1]{\textch_{{#1}}}
\newcommand{\bfv}{\mathbf{v}}
\newcommand{\bfw}{\mathbf{w}}

\DeclareMathOperator{\Coh}{Coh}
\DeclareMathOperator{\Per}{Per}

\newcommand{\sX}{\mathcal{X}}       %stack X

\newcommand{\sA}{\mathcal{A}}
\newcommand{\sB}{\mathcal{B}}

\newcommand{\sO}{\mathcal{O}}
\newcommand{\sK}{\mathcal{K}}
\newcommand{\sL}{\mathcal{L}}

\newcommand{\sT}{\mathcal{T}}

\newcommand{\sC}{\mathcal{C}}
\newcommand{\sD}{\mathcal{D}}

\newcommand{\sF}{\mathcal{F}}

\newcommand{\PP}{\mathcal{P}}

\newcommand{\st}{\;|\;}					%such that
\usepackage{leftidx}						%allows sub/superscripts on the left of the symbol

\newcommand{\AAA}[1]{\textcolor{red}{#1}}

\newcommand{\quotient}[2]{{\raisebox{.2em}{$#1$}\left/\raisebox{-.2em}{$#2$}\right.}}

\DeclarePairedDelimiter{\set}{\lbrace}{\rbrace}

\DeclarePairedDelimiter{\pair}{\langle}{\rangle}
\DeclarePairedDelimiter{\norm}{\lVert}{\rVert}
\DeclarePairedDelimiter{\abs}{\lvert}{\rvert}

\title{The stability manifold of local orbifold elliptic quotients}

\author[F. Rota]{Franco Rota}
\address{FR: School of Mathematics and Statistics, University of Glasgow, Glasgow G12 8QQ, United Kingdom} 
\email{franco.rota@glasgow.ac.uk}

\thanks{The author was partially supported by NSF-FRG grant DMS 1663813 and by EPSRC grant EP/R034826/1}

\begin{document}

\begin{abstract}
We study the stability manifold of local models of orbifold quotients of elliptic curves. In particular, we show that a region of the stability manifold is a covering space of the regular set of the Tits cone of the associated elliptic root system. The construction requires an explicit description of the McKay correspondence \cite{BKR01} for $A_N$ surface singularities and a study of wall-crossing phenomena. 
\end{abstract}
\maketitle

%\setcounter{tocdepth}{1}
%\tableofcontents

%\input{intro.tex}

%\mainmatter
\section{Introduction}

The space of stability conditions on a triangulated category $\sD$ was introduced by Bridgeland in \cite{Bri07_triang_cat}, following work of Douglas on $\Pi$-stability in string theory \cite{Dou02}.  Bridgeland shows that the set of these stability conditions is a complex manifold $\Stab(\sD)$ \cite{Bri07_triang_cat}, equipped with a local isomorphism 
$$\pi\colon \Stab(\sD) \to \Hom(K(\sD),\C).$$

The stability manifold $\Stab(\sD)$ is fully understood in the case when $\sD$ is the derived category of coherent sheaves on a smooth projective curve (see \cite{Bri07_triang_cat} for the elliptic curve, \cite{Mac07} for curves of positive genus, and \cite{BMW15}, \cite{Oka06} for the projective line). In the case of an elliptic curve, the stability manifold acquires a mirror-symmetric significance, in fact, it can be expressed as a $\C^*$-bundle over the modular curve \cite{Bri09_spaces}.

In this paper, we show that a similar interpretation is possible for quotients of elliptic curves by a group of automorphisms. We work on surfaces and describe the stability manifold for $\sD$ a certain triangulated category on a local model of an elliptic quotient. The main result of this paper is Theorem \ref{thm_thm1.1}, which expresses $\Stab(\sD)$ as a covering space of a subset of $\Hom(K(\sD),\C)$ determined by the data of the quotient. Theorem \ref{thm_thm1.1} represents an extension of previous results in two directions: on the one hand, it is an analog of the work of Bridgeland and Thomas on Kleinian singularities \cite{Bri09_kleinian}, \cite{Tho06} in the context of simple elliptic singularities. At the same time, it extends Ikeda's result \cite{Ike14} on arbitrary root systems of symmetric Kac-Moody Lie algebras to the case of elliptic root systems.

%The stability manifold is fully understood in the case when $\sD$ is the derived category of coherent sheaves on a smooth projective curve (see \cite{Bri07_triang_cat} for the elliptic curve, \cite{Mac07} for curves of positive genus, and \cite{BMW15}, \cite{Oka06} for the projective line). In the case that $E$ is an elliptic curve, the stability manifold acquires a mirror-symmetric interpretation, in fact, it can be expressed as a $\C^*$-bundle over the modular curve \cite{Bri09_spaces}. 

\subsection*{Summary of the results}

Let $X$ be the orbifold quotient of an elliptic curve $E$ by a group of group automorphisms of $E$. The orbifold $X$ is a weighted projective line of genus $0$ in the sense of Geigle and Lenzing \cite{GL87}. We consider its local model; in other words, we embed $X$ as the zero section in the total space of its cotangent bundle $Y\coloneqq\Tot(\omega_X)$, and let $\sD$ be the triangulated subcategory of $D^b(\Coh(Y))$ generated by sheaves supported on $X$.

Studying $\sD$, rather than $D^b(X)$, has two main advantages: the elliptic root system associated with $X$ is more evident, and one can use the McKay correspondence to compare the local orbifold to a smooth surface. From this point of view, local orbifold elliptic quotients represent an analog of Kleinian singularities.

%Every such quotient has $\pr 1$ as coarse moduli space, and it has $p_1,...,p_n$ orbifold points with stabilizers $\mu_{r_i}$ at the point $p_i$, we denote it by $\pr 1_{r_1,...,r_n}$. Over the field of complex numbers, there are only two possibilities for special automorphism groups, namely $\Z/4\Z$ and $\Z/6\Z$. These give rise to three possible quotients: $\pr {1}_{3,3,3}$, $\pr 1_{4,4,2}$ and $\pr 1_{6,3,2}$. 

The space $\Hom(K(\sD),\C)$ can be given a representation-theoretic interpretation as follows. The bilinear Euler form $\chi\colon K(\sD)\times K(\sD) \to \Z$ defined as 
$$ \chi(E,F)\coloneqq \sum\limits_{i=0}^\infty (-1)^i \dim_\C\Hom_\sD(E,F[i]) $$
is symmetric since $\sD$ is a K3-category, and $K(\sD)$ is identified with the root lattice of an elliptic root system $R$,
% of type $E_s^{(1,1)}$, $s=6,7,8$, 
whose bilinear form matches the Euler form. This premise is similar to Bridgeland's in \cite{Bri09_kleinian}, with the difference that $\chi$ is only negative semindefinite here. We denote by $a\coloneqq -[\sO_x]$ and $b \coloneqq\sum_{i=0}^2[\omega_X^{\otimes i}]$ the two classes generating its radical.

The Weyl group $W$ on $\Hom(K(\sD),\C)$ acts on the region 
$$\mathbb{E}\coloneqq \left\lbrace Z\in \Hom(K(\sD),\C) \st Z(a)=1, \im Z(b)>0 \right\rbrace, $$
which coincides with the Tits cone of the affine root system $R_a = R / \Z a$ (Lemma \ref{lem_E_is_Tits}).
Let $D$ be a fundamental domain for the action of $W$ on $\mathbb{E}$. We exhibit a region $U$ in the stability manifold which is homeomorphic to $D$ (Prop. \ref{prop_FundamentalRegionU}) and lift the action of $W$ using a group $\Br(\sD)$ of autoequivalences of $\sD$, generated by spherical twists (as defined by Seidel and Thomas \cite{ST01}).

A key step in the construction of $U$ is the McKay correspondence \cite{BKR01}: it gives an equivalence of categories between $D^b(Y)$ and the minimal resolution $Y'$ of the coarse space of $Y$. %\to \underline{\Tot(\omega_X)}$
In turn, this induces an equivalence between $\sD$ and the triangulated category $\sD'$ generated by sheaves supported on the pull-back of the zero section to $Y'$. We define a heart of a bounded t-structure $\mathcal{A}_R\subset\sD$ as the inverse image of $\Coh(Y')\cap \sD'\subset \sD'$. 
Then, we use the relation between coherent sheaves on $Y$ and perverse sheaves on $Y'$ (see \cite{Bri02_flops,VdB04_flops}) to explicitly describe $\sA_R$, classify its objects (Prop. \ref{prop_classification-objects-A}), and finally define $U$ as the region of $\Stab(\sD)$ containing conditions $(Z,\mathcal{A}_R)$ with $Z\in D$.
%\AAA{rewrite from here: need to intrduce a and b}
Denote by $\Stab^\dagger(\sD)$) the connected component of $\Stab(\sD)$ containing $U$.
 
We will often restrict our attention to the locus of \textit{normalized} stability conditions 
$$ \Stab_n(\sD)\coloneqq\{ \sigma=(Z,\PP) \in \Stab^\dagger(\sD) \st Z(a)=1 \} $$
rather than the full $\Stab^\dagger(\sD)$, and we let $\Stab_n^\dagger(\sD)$ be the connected component of $\Stab_n(\sD)$ containing $U$.
Normalization is a natural approach, effective in the study of threefold singularities (see for example \cite{Tod08_crepant}) and fitting with the representation-theoretic definition of $\mathbb{E}$. Moreover, every stability condition in $\Stab^\dagger(\sD)$ is obtained from $\Stab_n(\sD)$ using the natural $\C$-action (see Remark \ref{rmk_normalization}).

We show in Prop. \ref{prop_circ=dagger} that the condition 
\[ (*)\colon \quad  \im \frac{Z(b)}{Z(a)}>0\]
is automatic for all stability conditions in $\Stab^\dagger(\sD)$ (and hence in $\Stab_n^\dagger(\sD)$), and therefore $\pi$ maps $\Stab_n^\dagger(\sD)$ to $\mathbb{E}$.
The proof requires to understand wall-crossing for some specific classes in $K(\sD)$, which we do in Section \ref{sec_WallCrossing}. Our wall-crossing result can be viewed as a local analog of the classification of indecomposable sheaves on $X$ by Lenzing and Meltzer \cite[Theor. 4.6]{LM93}:

\begin{thm}[ = \ref{thm_enough_spherical_objects} ]
\label{thm_wall_crossing}
Let $\alpha$ be a root in the elliptic root lattice $K(\sD)$, and let $\sigma\in\Stab^\dagger(\sD)$ be generic with respect to $\alpha$. Then, there exists a $\sigma$-stable object $E$ of class $\alpha$. The object $E$ is rigid if $\alpha$ is a real root, and it varies in a family if $\alpha$ is imaginary.
\end{thm}

%Let $\Stab^\dagger(\sD)$ be the connected component of $\Stab(\sD)$ containing $U$. In addition to the full stability manifold $\Stab(\sD)$, we will often restrict our attention to the locus of \textit{normalized} stability conditions 
%$$ \Stab_n(\sD)\coloneqq\{ \sigma=(Z,\PP) \in \Stab(\sD) \st Z(a)=1 \}. $$

%By construction, $U\subset \Stab_n(\sD)$, so we also define $\Stab_n^\dagger(\sD)\subset \Stab^\dagger(\sD)$ as the connected component of $\Stab_n(\sD)$ containing $U$. We use $\pi$ to denote the restriction of the central charge map to any of these regions of $\Stab(\sD)$. 

%As it turns out, we have 
%\begin{proposition}\label{prop_circ=dagger}
%All stability conditions in $\Stab^\dagger(\sD)$ (and hence in $\Stab^\dagger_n(\sD)$) satisfy the additional condition 
%\[ (*)\colon \quad  \im \frac{Z(b)}{Z(a)}>0.\]
%\end{proposition}

%We assume Proposition \ref{prop_circ=dagger} for now, and we prove it in Section \ref{sec_proof_dagger=circ}, using our wall-crossing result (Theorem \ref{thm_wall_crossing}), which does not depend on the rest of the present Section \ref{sec_stabCondOnD}. 

%\AAA{go on from here}
%Rather than considering the image of the whole component $\Stab^\dagger(\sD)$ in $\Hom(K(\sD),\C)$, we restrict to stability conditions whose central charge $Z$ is normalized, i.e. 

The image of $\Stab_n^\dagger(\sD)$ is the set of regular orbits of $W$ in $\mathbb{E}$, denoted $\mathsf{X}_{\mathrm{reg}}$ (Prop. \ref{prop_image_of_StabDagger}). Moreover, the action of $\Br(\sD)$ preserves $\Stab_n^\dagger(\sD)$, and $U$ is a fundamental domain for this action. 
This leads to the main result of this paper (analogous to \cite[Theorem 1.1]{Bri09_kleinian} and \cite[Theorem 1.1]{Ike14}):
 
\begin{thm}[ = \ref{thm_mainThm} ]
\label{thm_thm1.1}
There is a covering map
$$\bar\pi\colon\Stab_n^\dagger(\sD) \to \mathsf{X}_{\mathrm{reg}}/W,$$
and the group $\Br(\sD)$ acts as group of deck transformations.
\end{thm}

Let $\Aut^\dagger(\sD)\subset \Aut(\sD)$ be the subgroup of autoequivalences preserving the component $\Stab_n^\dagger(\sD)$. Write $\Aut^\dagger_*(\sD)$ for the quotient of $\Aut^\dagger(\sD)$ by the subgroup of autoequivalences which act trivially on $\Stab_n^\dagger(\sD)$. We also show, in analogy with \cite[Cor. 1.4]{Bri09_kleinian}:

\begin{corollary}[ = \ref{cor_autoequivalences} ]
There is an isomorphism
$$\Aut^\dagger_*(\sD) \simeq \Br(\sD) \rtimes \Aut(\Gamma), $$
Where $Aut(\Gamma)$ acts on $\Br(\sD)$ by permuting the generators. 
\end{corollary}

Observe that, unlike in \cite{Bri09_kleinian}, the shift functor does not belong to $\Aut^\dagger(\sD)$, since it does not preserve $\Stab^\dagger_n(\sD)$.

\subsection*{Remarks and further problems} 

\begin{remark}[Representation theory]
From the point of view of representation theory, the categories $\sD$ discussed here are equivalent to the CY-2 completions of Ringel's canonical algebras (see \cite{STW16}). 
%\vspace{-.1in}
\end{remark}

\begin{remark}[Mirror symmetry]\label{rmk_mirrsymm}
Theorem \ref{thm_thm1.1} can be interpreted as an istance of the same principle outlined in \cite{Bri09_spaces} for elliptic curves.
 
The general automorphism group of an elliptic curve $E$ is $\Z/2\Z$, generated by the involution $\iota$. 
Over the field of complex numbers, there are only two possibilities for special automorphism groups, namely $\Z/4\Z$ and $\Z/6\Z$. These give rise to three possible quotients: $\pr {1}_{3,3,3}$, $\pr 1_{4,4,2}$ and $\pr 1_{6,3,2}$, whose mirror partners are the \emph{simple elliptic singularities} $E_6^{(1,1)},E_7^{(1,1)}$, and $E_8^{(1,1)}$ \cite{KS11},\cite{MR11}. To these singularities, Saito associates a \emph{universal unfolding space} and an elliptic root system \cite{EARSI}. If $X$ is one of these quotients, a hyperbolic extension of $\mathsf{X}_{\mathrm{reg}}/ W$ is the universal unfolding of the mirror elliptic singularity. Thus, Theorem \ref{thm_thm1.1} details the relation between the unfolding spaces and the stability manifold and gives a partial answer to Conjecture 1.3 in \cite{STW16}.
%\vspace{-.1in}

The automorphism group of a general elliptic curve $E$ is generated by its involution $\iota$. Theorems \ref{thm_wall_crossing}  and \ref{thm_thm1.1} hold for $X=[E/\iota]$, however, a mirror-symmetric interpretation seems less clear in this case.

\end{remark}

As in \cite{Bri08_k3}, \cite{Bri09_kleinian}, we expect the following properties:
\begin{conjecture}\label{conj_usual_conjectures}
\begin{enumerate}[(i)]
    \item The space $\Stab(\sD)$ is connected, so that $\Stab(\sD)=\Stab^\dagger(\sD)$;
    \item the space $\Stab_n(\sD)$ is simply connected. This would also show that the Artin group $G_W\simeq\pi_1(\mathsf{X}_{\mathrm{reg}}/W)$ (see Proposition \ref{prop_FundamentalGroupGW}) is isomorphic to $\Br(\sD)$.
\end{enumerate}
\end{conjecture}

See \cite{Ike14} and references therein for progress on Conjecture \ref{conj_usual_conjectures} in related frameworks.

\subsection*{Structure of the paper}
Section \ref{sec_StabCond} contains preliminaries on Bridgeland stability conditions, and Section \ref{sec_ell_root_syst} recalls the main aspects of the theory of elliptic root systems. In Section \ref{sec_tri_cat_loc_ell_quot}, we introduce the triangulated category $\sD$ (\ref{ssec_root_system_associated_to_D}) construct the heart $\sA_R$ (\ref{ssec_PerverseSheaves}), classify its objects (\ref{ssec_ClassificationA}) and use it to construct $U$ (\ref{ssec_FundRegionU}). Section \ref{sec_WallCrossing} contains our wall-crossing result, and in Section \ref{sec_stabCondOnD} we prove the main result.

\subsection*{Conventions}
We work over the field $\C$ of complex numbers. All abelian and triangulated categories are assumed to be $\C$-linear. Given a graph $\Gamma$, we write $\abs{\Gamma}$ to denote the set of its vertices.

\subsection*{Acknowledgements}
I wish to thank my doctoral advisor, Aaron Bertram, for his guidance and enthusiasm in suggesting this problem. I am grateful to Bronson Lim and Huachen Chen for our fruitful discussions, and to Arend Bayer for his helpful comments on a preliminary version of this work. I thank Michael Wemyss for his advice, and also for his help with Lemma \ref{lem_Phi(S_m)=S_m}.

\section{Stability conditions}\label{sec_StabCond}

Stability conditions on triangulated categories were first introduced  by Bridgeland  and were inspired by work of Douglas on string theory (see \cite{Bri07_triang_cat} and references therein). We recall here the definition and basic properties of stability conditions and the stability manifold. We refer the interested reader to the seminal work of Bridgeland \cite{Bri07_triang_cat}, \cite{Bri08_k3} and to the surveys \cite{Huy14}, \cite{MS17}.

In what follows, $\mathbb{T}$ is a triangulated category, with Grothendieck group $K(\mathbb{T})$.

\begin{definition}
A \emph{slicing} of $\mathbb{T}$ is a collection $\PP=\set{\PP(\phi)}_{\phi\in\R}$ of full additive subcategories of $\mathbb{T}$ satisfying the following properties:
\begin{enumerate}[(i)]
\item $\Hom(\PP(\phi_1),\PP(\phi_2))=0$ for $\phi_1<\phi_2$;
\item for all $E\in \mathbb{T}$ there are real numbers $\phi_1>...>\phi_m$, objects $E_i\in \mathbb{T}$ and a collection of triangles 
\begin{equation*}
\begin{tikzcd}[column sep=1.1em]
0=E_0 \arrow{rr}  & & E_1 \arrow[dl]\arrow{rr}  & & E_2 \arrow[dl] \arrow{rr}  & & ...\arrow{rr}  & & E_{m-1}\arrow{rr}  & & E_m=E\arrow[dl] \\
 & A_1 \arrow[ul, dashed] & & A_2 \arrow[ul, dashed] & & & &  & & A_m \arrow[ul, dashed] &
\end{tikzcd}
\end{equation*}
where $A_i\in\PP(\phi_i)$; 
\item $\PP(\phi)[1]=\PP(\phi+1)$.
\end{enumerate}
\end{definition}

The extremes $\phi_1$ and $\phi_m$ are denoted $\phi^+(E)$ and $\phi^-(E)$ respectively. Given a slicing $\PP$, for $\alpha\leq\beta\in\R$ we denote by $\PP((\alpha,\beta))$ the extension closure of the subcategories $\set{\PP(\phi)\,\colon\, \phi\in (\alpha,\beta)}$ (similar definitions work for other intervals in $\R$).

%comment: choice of the lattice
\begin{comment}
The definition of a stability condition depends on some additional data. We fix a finite rank lattice $\Lambda$ and a surjective group homomorphism
$$ v\colon K(\sX)\to \Lambda. $$
Suppose moreover that $\norm{\cdot}$ is a norm on $\Lambda_\R$. Since all norms on $\Lambda_\R$ are equivalent, the following definitions don't depend on this choice.

The main example for us will be that of \emph{numerical} stability conditions, for which we take $\Lambda$ to be the image of the Chern character
$$\cch\colon K(\sX) \to  H^*(\sX,\Q)$$
and $v$ is the Chern character map. We'll use the notation $\Stab(\sX)$ to denote the set of numerical stability conditions.
\end{comment}

\begin{definition}\label{def_stability condition}
A \emph{stability condition} on $\mathbb{T}$ is a pair $\sigma=(Z,\PP)$ where:
\begin{enumerate}[(i)]
\item $\PP$ is a slicing of $\mathbb{T}$;
\item $Z\colon K(\mathbb{T}) \to \C$ is an additive homomorphism called the \emph{central charge};
\end{enumerate}
and they satisfy the following properties:
\begin{enumerate}
\item For any non-zero $E\in \PP(\phi)$, $$ Z([E])\in \R_{>0} \cdot e^{i\pi\phi}; $$
\item (Support property) Fix any norm $\norm{\cdot}$ on $K(\mathbb{T})$. Then we require $$\inf\left\lbrace \dfrac{\abs{Z([E])}}{\norm{[E]}} \,\colon\, 0\neq E\in\PP(\phi),\,\phi\in\R \right\rbrace >0. $$
\end{enumerate}
\end{definition}

Given a stability condition $\sigma=(Z,\PP)$, we'll refer to $\PP((0,1])$ as to the \emph{heart} associated to $\sigma$. In fact, $\PP((\alpha,\alpha +1])$ is always the heart of a bounded $t$-structure for all $\alpha\in\R$, and it's an abelian category. 

If $E\in \PP((\alpha,\alpha +1])$ for some $\alpha\in\R$, then we say that $E$ has \emph{phase} $\phi$ if $Z([E])\in \R_{>0} \cdot e^{i\pi\phi}$, for $\phi\in(\alpha,\alpha+1]$. The nonzero objects of $\PP(\phi)$ are said to be $\sigma$\emph{-semistable} of phase $\phi$, and the simple objects of $\PP(\phi)$ are said to be $\sigma$\emph{-stable}.

For the general theory about bounded $t$-structures, we refer the reader to \cite{BBD82}, here we only recall the following lemma, which will be useful in what follows.

\begin{lemma}\label{lem_HeartsContainEachOther}
Let $\sA,\sB\subset \mathbb{T}$ be hearts of bounded t-structures on a triangulated category $\mathbb{T}$. If $\sA\subset\sB$,
then $\sA = \sB$.
\end{lemma}
\begin{proof}
This is \cite[Ex. 5.6]{MS17}.
\end{proof}

\begin{remark}[{\cite[Prop. 5.3]{Bri07_triang_cat}}] To construct stability conditions it is often convenient to use an alternative definition. In fact, sometimes we will write a stability condition as a pair $\sigma=(Z,\sA)$, where $\sA$ is the heart of a bounded $t$-structure and $Z$ is a \emph{stability function} satisfying Harder-Narasimhan and support property. A stability function is a linear map $Z\colon K(\sA)\to \C$ such that any non-zero $E\in \sA$ satisfies $ Z([E])\in \R_{>0} \cdot e^{i\pi\phi}$ with $\phi\in(0,1]$. Then one defines $\phi$ to be the phase of $E$, and declares $E$ to be $\sigma$-(semi)stable if for all non-zero subobjects $F\in\sA$ of $E$, $\phi(F)<(\leq) \phi(E)$. We say that $Z$ satisfies the HN property if for every $E\in\sA$ there is a unique filtration
$$ 0=E_0\subset E_1\subset ... \subset E_{n-1}\subset E_n=E $$
such that the quotients $E_i/E_{i-1}$ are $\sigma$-semistable of phases $\phi_i=\phi(E_i/E_{i-1})$, $\phi_1>\phi_2>...>\phi_n$. The support property is the same as in Definition \ref{def_stability condition}.
To recover a slicing as in Definition \ref{def_stability condition}, set $\PP(\phi)$ to be the category of $\sigma$-semistable objects of phase $\phi$ for $\phi\in (0,1]$, and declare $\PP(\phi)=\PP(\phi + n)$ for all $n\in \Z$.
\end{remark}

The following proposition is a useful tool to check the Harder-Narasimhan property:

\begin{proposition}[{\cite[Prop. 4.10]{MS17}}]\label{prop_BriHNproperty}
Suppose $\sA$ is an abelian category, and $Z\colon K(\sA)\to \C$ is a stability function. If
\begin{enumerate}[(i)]
\item the category $\sA$ is noetherian, and
\item the image of $\im Z$ is discrete in $\R$,
\end{enumerate}
then $Z$ has the Harder-Narasimhan property.
\end{proposition}

\begin{comment}
\begin{proposition}[{\cite[Prop. 2.4]{Bri07_triang_cat}}]
Suppose $\sA$ is an abelian category, and $Z\colon K(\sA)\to \C$ is a stability function satisfying the chain conditions:
\begin{enumerate}[(i)]
\item there are no infinite sequences 
$$\dots\subset E_{j+1}\subset E_j\subset \dots \subset E_1$$
of subobjects in $\sA$, with $\phi(E_{j+1})>\phi(E_j)$ for all $j$;
\item there are no infinite sequences 
$$E_1\twoheadrightarrow \dots\twoheadrightarrow E_{j}\twoheadrightarrow E_{j+1}\twoheadrightarrow \dots$$
of quotients in $\sA$, with $\phi(E_{j+1})<\phi(E_j)$ for all $j$.
\end{enumerate}
The $Z$ has the Harder-Narasimhan property.
\end{proposition}
\end{comment}

\subsection{The Stability manifold}

Let $\Stab(\mathbb{T})$ denote the set of stability conditions on $\mathbb{T}$. In \cite[Sec. 6]{Bri07_triang_cat}, Bridgeland shows that the function
\begin{equation}\label{eq_SlicingMetric}
f(\sigma,\tau)=\sup_{0\neq E\in \mathbb{T}}\set{\abs{\phi^+_\sigma(E)-\phi^+_\tau(E)},\abs{\phi^-_\sigma(E)-\phi^-_\tau(E)}}
\end{equation}
determines a generalized metric on $\Stab(\mathbb{T})$ which makes it into a topological space. Moreover, the central charge map $\pi\colon\Stab(\mathbb{T})\to \Hom(K(\mathbb{T}),\C)$ given by $(Z,\PP)\mapsto Z$ is a local homeomorphism, and it makes $\Stab(\mathbb{T})$ into a complex manifold of dimension $\rk(K(\mathbb{T}))$\cite[Thm. 1.2]{Bri07_triang_cat}.
The following lemma which will be useful later:

\begin{lemma}[{\cite[Lemma 6.4]{Bri07_triang_cat}}]\label{lem_BriCloseAreEqual}
Let $\sigma$, $\tau\in\Stab(\mathbb{T})$ be stability conditions with $\pi(\sigma)=\pi(\tau)$. If $f(\sigma,\tau)<1$, then $\sigma=\tau$.
\end{lemma}

Next, we recall two group actions on $\Stab(\mathbb{T})$. The additive group $\C$ acts as follows: for $z=x+iy\in \C$, define $z \cdot (Z,\PP) = (Z',\PP')$, with
\begin{equation}
\label{eq_C_action}
Z'(-) = e^{-z}Z(-), \qquad \PP'(\phi)=\PP\left(\phi+\frac{y}{\pi}\right). 
\end{equation}

The group of autoequivalences $\Aut(\mathbb{T})$ also acts on $\Stab(\sD)$: for $\Phi\in \Aut(\sD)$ and $\sigma=(Z,\PP)\in \Stab(\sD)$, define $\Phi\cdot (Z,\PP)=(Z',\PP')$ as the stability condition with
\begin{equation}
\label{eq_Aut_Action}
Z'(E)\coloneqq Z(\Phi^{-1} (E)) \;\text{ and }\; \PP'(\phi)\coloneqq \Phi\PP(\phi).
\end{equation}

\begin{comment}
A part of this work will be dedicated to the study of the map $\pi$. This will require a local analysis and some of the technical results from \cite{Bri07_triang_cat}. We recollect them here. The finiteness condition used in \cite{Bri07_triang_cat} is replaced and implied by the support property, see \cite[Sec. 4]{Huy14}).

\begin{lemma}[{\cite[Lemma 6.4]{Bri07_triang_cat}}]\label{lem_BriCloseAreEqual}
Let $\sigma$, $\tau\in\Stab(\mathbb{T})$ be stability conditions with $\pi(\sigma)=\pi(\tau)$. If $f(\sigma,\tau)<1$, then $\sigma=\tau$.
\end{lemma}

\begin{thm}[{\cite[Thm. 7.1]{Bri07_triang_cat}}]\label{thm_BriDeformingCentralCharge}
Let $\sigma=(Z,\PP)\in\Stab(\mathbb{T})$, and $0<\epsilon<\frac18$. If $Z'\in \Hom(K(\mathbb{T}),\C)$ satisfies 
$$ \abs{Z'(E)-Z(E)} < \sin(\pi\epsilon)\abs{Z(E)} $$ 
for every $\sigma$-stable in $\mathbb{T}$, then there exists a stability condition $\tau=(Z',\PP')\in\Stab(\mathbb{T})$ with $f(\sigma,\tau)<\epsilon$.
\end{thm}
\end{comment}

\subsection{Torsion pairs and tilts of abelian categories}

Next, we recall the definition of a \emph{tilt} of an abelian category $\sA$, which is a technique to produce new abelian subcategories of $D^b(\sA)$. Indeed, the tilt of a heart of a bounded t-structure is a new heart in $D^b(\sA)$ \cite{HRS96}.

\begin{definition}
Let $\mathcal A$ be an abelian category. A \emph{torsion pair} for $\mathcal A$ is a pair of full subcategories $(\sT,\sF)$ such that:
\begin{enumerate}[(i)]
\item $\HOM{\sT}{\sF}{}=0$;
\item for any $E\in \mathcal A$ there exists a short exact sequence
$$ 0\to T \to E \to F \to 0 $$
where $T\in \sT$ and $F\in \sF$. 
\end{enumerate}
%One often refers to $\sT$ (resp. $\sF$) as the \emph{torsion} part (resp. \emph{torsion-free} part) of the torsion pair. 
\end{definition}

Given a torsion pair $(\sT,\sF)$ on an abelian category $\sA$, we define $\sA^\sharp=\pair{\sF[1],\sT}$ to be the \emph{extension closure} of $\sF[1]$ and $\sT$, i.e. smallest full subcategory of $D^b(\sA)$ containing $\sF[1]$ and $\sT$ closed under extensions. $\sA^\sharp$ is called the \emph{tilt} of $\sA$ along the torsion pair $(\sT,\sF)$. Sometimes we will also refer to $\sA^\sharp[-1]=\pair{\sF,\sT[-1]}$ as to the tilt, but no confusion should arise.

\section{Elliptic root systems}\label{sec_ell_root_syst}

In this section we introduce elliptic root systems and recall some of their properties. Elliptic root systems were introduced by Saito \cite{EARSI,EARSII}, in our exposition we draw also from \cite{STW16} and \cite{IY17}.  

\begin{definition}[{\cite[Def. 1]{EARSI}}]\label{def_gen_root_syst}
Let $F$ be a real vector space of rank $l+2$, equipped with a positive semidefinite symmetric bilinear form $I\colon F\times F \to F$, whose radical $\rad I$ has rank 2. An \emph{elliptic root system} associated to $(F,I)$ is a subset $R\subset F$ of non-isotropic elements such that:

\begin{enumerate}
    \item the additive group generated by $R$, denoted $Q(R)$, is a full sublattice of $F$. That is, the embedding $Q(R)\subset F$ induces an isomorphism $Q(R)_\R\simeq F$;
    \item the form $I$ takes integer values on  $R\times R$;
    \item For all $\alpha$ in $R$, the reflection
    $$ w_\alpha (x) = x-I(x,\alpha)\alpha \,\text{ for }\,x\in F$$
    satisfies $w_\alpha(R)=R$;
    \item if $R=R_1\cup R_2$ with $R_1 \perp R_2$, then either $R_1$ or $R_2$ is empty.
\end{enumerate}
\end{definition}

The subgroup $W$ of $\Aut(F,I)$ generated by the $w_\alpha$ for $\alpha\in R$ is called the \emph{Weyl group} of the root system $R$. The lattice $\rad I \cap Q(R)$ is full in the two-dimensional vector space $\rad I$. A \emph{marking} of $R$ is the choice of a 1-dimensional subspace $G\subset \rad I$, and $(R,G)$ is called a \emph{marked elliptic root system}.

\begin{comment}
a subset $\Delta_{re}\subset R$ of \emph{real roots} such that:
\begin{enumerate}[(a)]
\item $R=\Z\Delta_{re}(R)$;
\item $I(\alpha,\alpha)=2$ for all $\alpha\in \Delta_{re}(R)$
\item for all $\alpha$ real roots, the elements $w_\alpha \in \Aut(R,I)$ defined by
      $$ w_\alpha(\lambda) \coloneqq \lambda - I(\lambda,\alpha)\alpha $$
      preserves $\Delta_{re}(R)$, i.e. $w_\alpha(\Delta_{re}(R))=\Delta_{re}(R)$. The group $$W(R)\coloneqq \pair{ w_\alpha \st \alpha \in \Delta_{re}(R)}$$is called the \emph{Weyl group} of $R$;
\item there exists a subset $B=\set{\alpha_0,...,\alpha_N}\subset \Delta_{re}(R)$ called a \emph{root basis} of $R$, satisfying $R=\oplus_i \Z\alpha_i$, $W(R)=\pair{\alpha_1,...,\alpha_N}$ and $\Delta_{re}(R)=W(R)B$;
\end{enumerate}
\item an element $c_R$ with a presentation $c_R=w_{\alpha_1}...w_{\alpha_N}$ for some root basis of $R$, called the \emph{Coxeter transformation} of $R$.
\begin{definition}
An elliptic root system $R$ is said to be \emph{oriented} if $\rad I$ is oriented. An \emph{admissible frame} of $\rad I$ is an oriented basis $(a,b)$ of $\rad I$ such that $Q(R)\cap \rad I\simeq \Z a\oplus\Z b$. Denote by $G$ the subspace $\R a\subset F$. In this case, we refer to the pair $(R,G)$ as to a \emph{marked} elliptic affine root system. We refer to $a$ as to a \emph{signed marking} of $R$.
\end{definition}
\end{comment}

To a marked elliptic root system we can associate an affine root system $R_a$ and a finite root system $R_f$ of rank $l$ by considering the quotients 
\begin{align*}
F_a\coloneqq F/G & & R_a\coloneqq R/R\cap G\\
F_f\coloneqq F/\rad I  & & R_f\coloneqq R/R\cap \rad I
\end{align*}
and the bilinear forms induced on $F_f$ and $F_a$ by $I$. 

Now fix a marked root system $(R,G)$, with generators $a,b$ for $\rad I\cap Q(R)$ and $G=\R a$. 

\begin{proposition}[{\cite[Cor. 2.3]{IY17}}]
The root system $R$ is given by
$$ R=\set{\alpha_f + mb + na \st \alpha_f\in R_f, m,n\in \Z}. $$
\end{proposition}

\begin{definition}[{\cite[\S 2.3]{IY17}}]
The elements of $R$ are also called the \emph{real roots} of $R$.
We define the set $\Delta_{im}$ of \emph{imaginary roots} of $R$ as 
$$\Delta_{im}=\left\lbrace mb+na \st m,n\in \Z\setminus \{0\}\right\rbrace. $$
%$$ \Delta_{re}=\set{\alpha\in R \st I(\alpha,\alpha)\neq 0} \qquad $$
\end{definition}

\subsection{The Dynkin graph}
\label{ssec_DynkinGraph}

To a marked elliptic affine root system $(R,G)$ one can associate a diagram $\Gamma_{R,G}$ called the \emph{Dynkin diagram} of $(R,G)$ (see \cite[\S 5]{EARSI}). In general, the vertices of $\Gamma_{R,G}$ are in bijection with a \emph{root basis} of $R$ (defined as in \cite[\S 3.4]{EARSI}), and two vertices $\alpha,\beta\in\abs{\Gamma_{R,G}}$ are connected following the rule:
    \begin{center}
\begin{tikzcd}[row sep=tiny]
\alpha & \beta &  \\
\circ   & \circ & \mbox{ if } \ \ I(\alpha,\beta)=0;  \\
\circ  \arrow[r,dash] & \circ &  \mbox{ if } \ \ I(\alpha,\beta)=-1;\\
\circ  \ar[equal, dashed]{r} & \circ & \mbox{ if } \ \ I(\alpha,\beta)=2.
\end{tikzcd}
\end{center}

The results of this section hold for all elliptic root systems (classified in \cite[Table 1]{EARSI}).

\begin{notation}\label{not_vertices_index}
In the rest of this work we will only need diagrams $\Gamma$ of the following specific shape (called an \textit{octopus} in \cite{STW16}):
 \begin{center}
\begin{tikzcd}[every arrow/.append style={dash}]
            &      &            &\circ \arrow[loop, out=90-30, in=90+30,looseness=6,draw=none]{}{(0,0)} \ar{dl} \ar[equal,dashed]{d} \ar{ddl} \ar{dr} \ar{ddr}  &&& \\
\circ \arrow[loop, out=90-30, in=90+30,looseness=6,draw=none]{}{(1,a_1-1)}\rar  & ... \rar & \circ \arrow[loop, out=70-30, in=70+30,looseness=6,draw=none]{}{(1,1)} \rar  &\circ \arrow[loop, out=270-30, in=270+30,looseness=6,draw=none]{}{(0,1)} \rar \ar{dr} \ar{dl} & \circ \arrow[loop, out=110-30, in=110+30,looseness=6,draw=none]{}{(r,1)} \rar & ... \rar & \circ \arrow[loop, out=90-30, in=90+30,looseness=6,draw=none]{}{(r,a_r-1)}\\
                         & & \circ \arrow[loop, out=70-30, in=70+30,looseness=6,draw=none]{}{(2,1)} \ar{dl} &      & \circ \arrow[loop, out=110-30, in=110+30,looseness=6,draw=none]{}{(r-1,1)} \ar{dr} & &  \\ 
                         &  ... \ar{dl} & &      & & ... \ar{dr} &  \\ 
                      \circ \arrow[loop, out=70-30, in=70+30,looseness=6,draw=none]{}{(2,a_2-1)}   &  & ... &   ...   & ... &  & \circ \arrow[loop, out=110-30, in=110+30,looseness=6,draw=none]{}{(r-1,a_{r-1}-1)}  \\ 
\end{tikzcd}
\end{center}
We assume from now on that elliptic diagrams have the octopus shape, and adopt the labelling shown above for the vertices of $\Gamma$. 
%When $\Gamma$ is an octopus-shaped diagram, we use indices $(0,0)$ and $(0,1)$ to label the two central vertices. The diagram $\Gamma'$ obtained by deleting $(0,0)$ and $(0,1)$ and all adjacent edges is a disjoint union of Dynkin diagrams of type $A_{a_i-1}$, with $i=1,...,r$. We use the index $(i,j)$ (with $j=1,...,a_i-1$) for the vertex  occupying the $j-th$ position on the $i-th$ diagram $A_{a_i-1}$. 
% ($j$ increases moving away from the center). 
%\AAA{review: use (0,0) and (0,1) for the two central vertices?}
We denote by $\alpha_{(i,j)}$ the root of $R$ corresponding to the vertex $(i,j)$.
% ( with $i=0$  and $j=0,1$ or $i=1,...,r$ and $j=1,...,a_i-1$). 
% We will use this indexing to label the generators $\alpha_{-1},...,\alpha_l$ when it is convenient.
\end{notation}

The marking of an octopus-shaped elliptic root system is generated by the class $a\coloneqq \alpha_{(0,1)}-\alpha_{(0,0)}$.
Erasing the $(0,0)$ vertex and all adjacent edges in the above diagrams yields the Dynkin diagram $\Gamma_a$ associated with $R_a$, so we have $\abs{\Gamma_a}= \abs{\Gamma}\setminus \{(0,0)\}$. Then $\{\alpha_{v}\}_{v\in \abs{\Gamma_a}}$ give a root basis for $R_a$. Let $b$
be the imaginary root of the affine system $R_a$ ($b$ is a positive linear combination of the $\set{\alpha_v}_{v\in{\abs\Gamma_a}}$, see \cite[Chap. 5]{Kac90}). Then, $(a,b)$ is a basis for $\rad I$.

\begin{example}\label{ex_elliptic_EARS}
Our main interest is in elliptic root systems arising from quotients of elliptic curves by automorphism groups (see Section \ref{ssec_root_system_associated_to_D}). They are the root systems of type $D_4^{(1,1)}$, $E_6^{(1,1)}$, $E_7^{(1,1)}$ and $E_8^{(1,1)}$, whose diagrams are all octopus-shaped:
\begin{center}
\begin{tikzcd}[row sep=tiny, every arrow/.append style={dash}]
                & \circ \rar \ar{ddr} &\circ \rar \arrow[ddl,crossing over] \ar[equal,dashed]{dd}  & \circ \\
D_4^{(1,1)}   &  & &   \\
&             \circ \rar    & \circ \rar  \arrow[uur]  & \circ \arrow[uul, crossing over]  

\end{tikzcd}
\end{center}
    \begin{center}
\begin{tikzcd}[row sep=tiny, every arrow/.append style={dash}]
             &   & &\circ \rar \ar{dl} \ar[equal,dashed]{dd}  & \circ \rar & \circ  \\
E_6^{(1,1)}  &\circ  \rar & \circ & & &   \\
&                & & \circ \rar \ar{ul} \arrow[uur]  & \circ \rar \arrow[uul, crossing over] &  \circ  

\end{tikzcd}
\end{center}
\begin{center}
\begin{tikzcd}[row sep=tiny, every arrow/.append style={dash}]
          & &\circ \rar \ar{dl} \ar[equal,dashed]{dd}  & \circ \rar & \circ \rar & \circ \\
E_7^{(1,1)}   & \circ & & & \\
&              & \circ \rar \ar{ul} \arrow[uur]  & \circ \rar \arrow[uul, crossing over] &  \circ \rar & \circ 
\end{tikzcd}
\end{center}
\begin{center}
\begin{tikzcd}[row sep=tiny, every arrow/.append style={dash}]
          &  &\circ \rar \ar{dl} \ar[equal,dashed]{dd}  & \circ \rar & \circ \rar & \circ \rar & \circ \rar & \circ \\
E_8^{(1,1)}   & \circ & & &   \\
&              & \circ \rar \ar{ul} \arrow[uur]  & \circ \rar \arrow[uul, crossing over] &  \circ
\end{tikzcd}
\end{center}
Erasing the $(0,0)$ vertex and all adjacent edges in the above diagrams yields the Dynkin diagrams of affine root systems of type $\tilde{D}_4, \tilde{E}_6,\tilde{E}_7$, and $\tilde{E}_8$ respectively. 
\end{example}

\begin{comment}
 but we recall some of the properties of $\Gamma_{R,G}$ which will be useful in what follows.

Let  $\{\alpha_0,\alpha_1,...,\alpha_l\}$ and $\{\alpha_1,...,\alpha_l\}$ be root bases of $R_a$ and $R_f$ respectively so that 
$$ F= \left(\oplus_{i=0}^l \R \alpha_i\right)\oplus G. $$
We define $\alpha_{-1}\coloneqq a-\alpha_0$.
\end{comment}

\subsection{The Weyl group} 

Since $a\in\rad I$, $W$ preserves the marking $G\subset F$. Then, the projection $p\colon F\to F/G$ induces a homomorphism $p_*\colon W\to W_a$ to the affine Weyl group associated with $R_a$. Denote by $T$ the kernel of $p_*$. 

\begin{lemma}[{\cite[(1.15)]{EARSI}}]\label{lem_Wa_in_W}
The subgroup of $W$ generated by $\set{w_{\alpha_v} \st v\in \abs{\Gamma_a}}$ is isomorphic to $W_a$, so the sequence 
\begin{equation}\label{eq_ses_TWWa}
0\to T\to W\to W_a \to 1    
\end{equation}
splits into a semi-direct product $W=T \rtimes W_a$.
\end{lemma}

Next we give an explicit descripton of $T$. To do so, we introduce the following elements of $W$:

\begin{definition}\label{def_r_ij}
For each vertex of $\Gamma_a$ define elements of $W$:
\begin{enumerate}
    \item $r_{(0,1)} \coloneqq w_{\alpha_{(0,1)}}w_{\alpha_{(0,0)}}$;
    \item $r_{(i,1)}\coloneqq w_{\alpha_{(i,1)}} r_{(0,1)} w_{\alpha_{(i,1)}} r_{(0,1)}^{-1}$ for $i=1,...,r$;
    \item $ r_{(i,j)}\coloneqq w_{\alpha_{(i,j)}} r_{(i,j-1)} w_{\alpha_{(i,j)}} r_{(i,j-1)}^{-1}$ for $i=1,...,r$, $j=2,...,a_i-1$;
\end{enumerate}
\end{definition}

\begin{lemma}[{\cite[Theor. 3.5]{STW16}}]\label{lem_TisQ(Rf)}
For $v\in\abs{\Gamma_a}$, let $\alpha_v$ be the corresponding root and $r_v$ the corresponding element from Def. \ref{def_r_ij}. For all $\beta\in F$, we have
$$ r_v(\beta)= \beta - I(\beta,\alpha_v)a. $$
Moreover, there is a group homomorphism
\begin{align*}
\varphi \colon Q(R_a) & \to W \\
    \sum_{v\in \abs{\Gamma_a}} m_{v} \alpha_{v} & \mapsto \prod_{v\in\abs{\Gamma_a}} r_{v}^{m_{v}}
\end{align*}
with kernel generated by $b$. The group $T$ is isomorphic to the lattice $\varphi (Q(R_a))\simeq Q(R_f)$, and $\varphi$ induces the inclusion $T\to W$ of the exact sequence \eqref{eq_ses_TWWa}.
\end{lemma}

\subsection{Tits cone, regular set, and fundamental domain}\label{sec_the_regular_set}
We follow \cite{EARSII} and define:
\begin{equation}\label{eq_def_EandH}
\begin{split}
\mathbb H \coloneqq \set{x\in\Hom(\rad I, \C) \st x(a)=1,\im x(b)>0};\\
\mathbb E \coloneqq \set{x\in\Hom(F, \C) \st x(a)=1,\im x(b)>0}.
\end{split}
\end{equation}

The Weyl group $W$ acts on $\mathbb E$ by $(gx)(\beta)\coloneqq x(g^{-1}\beta)$ for $x\in \mathbb E$ and $g\in W$. This action preserves $x_{|\rad I}$, so it respects the restriction map $s\colon\mathbb E \to \mathbb H $. 
%Let $\mathfrak{h}$ be a (complexified) Cartan subalgebra of the finite root system $R_f$. We may identify $\mathbb{E}=\mathbb{H}\times \mathfrak{h}$, and the restriction map coincides with the projection onto the first factor.

With the goal of describing a fundamental domain for the action of $W$ on $\mathbb E$, we will identify $\mathbb{E}$ with the complexified Tits cone of $R_a$ (see \cite[\S 3.12]{Kac90} for basic facts about Tits cones). 

Recall that to the affine root system $R_a$ is associated the Weyl alcove
$$ A_\R\coloneqq \set{h\in Q(R_a)_\R^* \st h(\alpha_v)> 0 \ \ \mbox{ for } v\in\abs{\Gamma_a} } $$
and the (real) \textit{Tits cone} $\mathsf T_{\R}(R_a)$, defined as the topological interior of
$$\overline{\mathsf T_\R(R_a)}\coloneqq \bigcup\limits_{w\in W_a} w \overline{A_\R}.  $$

The \emph{complexified Tits cone} associated to $R_a$ is 
$$ \mathsf T(R_a)\coloneqq \set{h\in Q(R_a)_\C^* \st \im h \in  \mathsf T_{\R}(R_a)}. $$

\begin{lemma}\label{lem_E_is_Tits}
There is an isomorphism of complex manifolds between $\mathbb{E}$ and $\mathsf T(R_a)$, equivariant with respect to the action of $W_a$. 
\end{lemma}

\begin{proof}
Consider the inclusion $Q(R_a)\subset Q(R)$ mapping $\alpha + \Z a\in R_a$ to $\alpha \in Q(R)$. This induces a restriction map $\phi\colon\Hom(Q(R),\C)\to \Hom(Q(R_a),\C)$.

The complexified Tits cone can be equivalently described as 
$$ \mathsf T(R_a)= \set{h\in Q(R_a)_\C^* \st \im h(b)>0} $$
(this is \cite[Lemma 2.12]{Ike14}). Then, it is clear that $\phi$ is a holomorphic map sending $\mathbb{E}$ bijectively onto $\mathsf T(R_a)$.
Moreover, the action of $W_a$ on $\mathsf T(R_a)$ coincides with that on $\mathbb{E}$ through $W_a \subset W$ as in Lemma \ref{lem_Wa_in_W}. 
\begin{comment}
To establish the claim. Fix $h\in \mathsf T(R_a)$. We have $wh(b)=h(w^{-1}b)=h(b)$ for all $w\in W_a$ since $b\in \rad I$, so we may assume $h\in A_\R$. The imaginary root $b$ is a positive linear combination of real roots $\alpha_v$ for $v\in \abs{\Gamma_a}$, therefore every $h\in \mathsf A_\R$ satisfies $\im h (b)>0$. 
Conversely, \cite[]{Kac90} 
(see the discussion in \cite[Section 2.3]{Ike14}).\end{comment}
\end{proof}

 In order to describe the action of $T$ on $\mathbb{E}$ (see Lemma \ref{lem_Wa_in_W}), it will be convenient to emphasize a complex structure on $\mathbb{E}_\tau \coloneqq s^{-1}(\tau)$ induced by $\tau\in \HH$.
In fact, $\tau$ defines an isomorphism $\rad I \simeq \C$ by 
$$  ua+vb \mapsto u + v\tau. $$
Next, identify $\mathbb{E}_\tau$ with the relative tangent space of $\pi$ over $\tau$. This is a complexification 
$$ V\otimes_\R \C \,\text{ where }\, V\coloneqq (F/\rad I)^*. $$
The bilinear form $I$ induces an isomorphim $ I^*\colon V\xrightarrow{\sim} V^*= F/\rad I$, and in turn an isomorphism of complex vector spaces 
$$\tau\otimes I\colon (F/\rad I)\otimes_\R \rad I \simeq V\otimes_\R \C.$$ 
We write
\begin{equation}\label{eq_iso_radFrad_to_Vc}
\mathbb{E}_\tau \simeq V\oplus \tau V.
\end{equation}
Then we have:

\begin{lemma}\label{lem_Saito_actions on Vc}
\begin{enumerate}[(i)]
    \item $W$ acts preserving fibers $\mathbb E_\tau$ above a point in $\tau\in\mathbb H$;
    \item Under the identification \eqref{eq_iso_radFrad_to_Vc}, the group $T$ acts as a finite index subgroup of the real translation lattice $Q(R_f)\subset V$. In particular, $T$ acts freely on $\mathbb{E}$.
\end{enumerate}
\end{lemma}

\begin{proof}
The first statement is straightforward, since $\tau$ is determined by the restriction of $x\in \mathbb{E}$ to $\rad I$, which is $W$-invariant.
The second statement follows immediately from Lemma \ref{lem_TisQ(Rf)} and the fact that $x(a)=1$ for all $x\in \mathbb{E}$. %The second statement is \cite[3.5.6]{EARSII}.
\end{proof}

We can finally describe the regular set for the action of $W$ on $\mathbb{E}$:

\begin{proposition}\label{prop_Saito_action_of_W}
The action of $W$ on $\mathbb E$ is properly discontinuous. Moreover,
the space of regular orbits of $W$ is $$ \mathsf{X}_{\mathrm{reg}}\coloneqq \mathbb E \setminus \cup_{\alpha\in R} H_\alpha,$$
where $H_\alpha\subset \mathbb{E}$ is the reflection hyperplane defined by the equation $x(\alpha)=0$. 
%\quad (\mbox{resp. } \mathsf{X}_{\mathrm{reg}}\coloneqq \mathbb E \setminus \cup_{\alpha\in R}H_\alpha )$$
\end{proposition}

\begin{proof}
The first statement is \cite[(3.5)]{EARSII}. The second follows from the description of the regular set of $\mathsf{T}(R_a)$ (\cite[Prop. 3.12]{Kac90}), combined with Lemma \ref{lem_E_is_Tits} and the fact that $T$ acts freely on $\mathbb{E}$ (Lemma \ref{lem_Saito_actions on Vc}).
\end{proof}

We think of $\mathsf{X}_{\mathrm{reg}}$ and $\mathbb{E}$ as naturally sitting in $\Hom(F,\C)$. 

Denote by $A\subset \mathsf T(R_a)$ the \textit{complexified Weyl alcove}
$$ A\coloneqq \set{h\in \mathsf{T}(R_a) \st \im h \in A_\R}. $$
We think of $A$ as embedded in $\mathbb{E}$ via Lemma \ref{lem_E_is_Tits}, and  write $A_\tau$ for the intersection of $A$ with $\mathbb E_\tau$.

Let $B'$ be a hypercube in $V$ which contains the origin and is a fundamental domain for the action of $T$ on $V$, and define $B_\tau\coloneqq \set{h\in \mathbb E_\tau \simeq V_\C \st \re (h) \in B'}$. 

\begin{proposition}\label{prop_def_fundamental_domain}
A fundamental domain for the action of $W$ on $\mathbb E_\tau$ is the intersection 
$$ D_\tau\coloneqq A_\tau\cap B_\tau.$$
A fundamental domain for the action of $W$ on $\mathbb E$ is $D\coloneqq \cup_{\tau\in \mathbb{H}} D_\tau \simeq D_{\sqrt{-1}}\times \HH\subset \mathsf{X}_{\mathrm{reg}}$.
\end{proposition}

\begin{proof}
As a consequence of Prop. \ref{prop_Saito_action_of_W}, it is enough to show that for every $Z\in\mathbb E_\tau$ there exists an element $w\in W$ such that $w\cdot Z\in D_\tau$.  Using the complex structure given in \eqref{eq_iso_radFrad_to_Vc}, we may write every $Z\in \mathbb E_\tau$ as $\re Z + \tau\im Z$. The closed alcove $\overline {A_\R}$ is a fundamental domain for the action of $W_a$ on $\overline{\mathsf T_\R(R_a)}$ \cite[Prop. 3.12]{Kac90}, so there exists an element $w'\in W_a$ such that $w' \cdot Z\in A_\tau$. By definition of $B_\tau$, there is an element $r\in T$ such that $r\cdot (w'\cdot Z) \in B_\tau$ and $\im r\cdot (w'\cdot Z)=\im (w'\cdot Z)$, so $r\cdot (w'\cdot Z)\in D_\tau$.

The statement about $\mathbb E$ follows, since every $w\in W$ preserves the fibers $\mathbb E_\tau$ by Lemma \ref{lem_Saito_actions on Vc}.
\end{proof}

\subsection{Boundary of \texorpdfstring{$D$}{D} and fundamental group} \label{sec_fund-dom-and-boundary}

Next, we describe the boundary of $D$ in $\mathsf{X}_{\mathrm{reg}}$ in terms of walls for the action of $W$. 
For vertices $v\in\abs{\Gamma_a}$ we define walls $W_{v,\pm}\subset \overline D$ for the Weyl alcove
\begin{align*}
    W_{v,+}\coloneqq \set{Z\in \mathsf{X}_{\mathrm{reg}}\cap \overline D \st Z(\alpha_v)\in \R_{>0}, \im Z(\alpha_w)>0 \mbox{ for }v\neq w \in \abs{\Gamma_a}}\\
     W_{v,-}\coloneqq \set{Z\in \mathsf{X}_{\mathrm{reg}}\cap \overline D \st Z(\alpha_v)\in \R_{<0}, \im Z(\alpha_w)>0 \mbox{ for }v\neq w \in \abs{\Gamma_a}}
\end{align*}
For vertices $u\in \abs{\Gamma_{f}}$, write $Y'_{u,\pm}$ for the faces of the fundamental hypercube $B'$, and let 
$$Y_{u,\pm}\coloneqq \left[\cup_\tau (Y'_{u,\pm}\oplus \tau V)\right]  \cap \overline D \subset \mathsf{X}_{\mathrm{reg}}. $$
Then, the boundary of $D$ in $\mathsf{X}_{\mathrm{reg}}$ is contained in the union of the walls $W_{v,\pm}$ and $Y_{u,\pm}$ as $v,u$ vary. 

Next, we describe the fundamental group of $\mathsf{X}_{\mathrm{reg}}/W$. 

\begin{definition}
Let $R$ be an elliptic root system. The Artin group $G_W$ associated with the Weyl group $W$ is the group generated by $\set{g_v,h_v\st v\in\abs{\Gamma_a}}$ with relations
\begin{align*}
    g_vg_u=g_ug_v &\;\mbox{ if }\; I(\alpha_v,\alpha_u)=0;\\
    g_vg_ug_v=g_ug_vg_u &\;\mbox{ if }\; I(\alpha_v,\alpha_u)=-1;\\
    h_vh_u=h_uh_v &\;\mbox{ for all }\; u,v\in\abs{\Gamma_a};\\
    g_vh_u=h_ug_v &\;\mbox{ if }\; I(\alpha_v,\alpha_u)=0;\\
    g_vh_ug_v=h_uh_v &\;\mbox{ if }\; I(\alpha_v,\alpha_u)=-1;\\
\end{align*}
\end{definition}

\begin{proposition}\label{prop_FundamentalGroupGW}
Suppose $R$ is an elliptic root system. Then, the fundamental group of $\mathsf{X}_{\mathrm{reg}}/W$ is 
$$ \pi_1(\mathsf{X}_{\mathrm{reg}}/W, *) \simeq G_W. $$
The generator $g_v$ of $G_W$ is given by the path connecting $*$ and  $w_{\alpha_v}(*)$ passing through $W_{v,+}$ just once. The generator $h_v$ of $G_W$ is given by the path connecting $*$ and $r_v(*)$ which is constant in the imaginary part. 
%$$ \pi_1(\mathsf{X}_{\mathrm{reg}}/W, *) \simeq \Z[\eta]\times G_W. $$
%The path $[\eta]$ corresponds to the $S^1$-orbit of $*$ in $\C^*$. The generator $g_v$ of $G_W$ is given by the path connecting $*$ and  $w_{\alpha_v}(*)$ passing through $W_{v,+}$ just once. The generator $h_v$ of $G_W$ is given by the path connecting $*$ and $r_v(*)$ which is constant in the imaginary part. 
\end{proposition}

\begin{proof}
%We have $\mathsf{X}_{\mathrm{reg}}\simeq \C^*\times \mathsf{X}_{\mathrm{reg}}^N$ and $\pi_1(\C^*)\simeq \Z$, so it is enough to show that $\pi_1(\mathsf{X}_{\mathrm{reg}}^N/W)\simeq G_W$. 
By Lemma \ref{lem_E_is_Tits}, the set $\mathsf{X}_{\mathrm{reg}}$ coincides with the regular subset of the complexified Tits cone $ \mathsf T_{\mathrm{reg}}(R_a)$. It is shown in \cite{vdL83} that $\pi_1(\mathsf T_{\mathrm{reg}}(R_a)/W)\simeq G_W$.
\end{proof}

\section{Triangulated categories associated to local elliptic quotients}\label{sec_tri_cat_loc_ell_quot}

We consider orbifold curves obtained from a quotient of an elliptic curve by a finite subgroup of its automorphism groups.
Every elliptic quotient has $\pr 1$ as coarse moduli space and orbifold points $p_i$ with stabilizers $\mu_{a_i}$. Up to permuting the $p_i$'s, there are only 4 possibilities, namely: $\pr 1_{2,2,2,2}$ (here, $r=4$ and $a_i=2$ for all $i$), $\pr {1}_{3,3,3}$, $\pr 1_{4,4,2}$ and $\pr 1_{6,3,2}$. We denote them respectively $X_2,X_3,X_4$ and $X_6$.

Each $X_k$ is realized as a quotient of an elliptic curve $E_k$ by a cyclic group $\mu_k$ of group automorphisms:
$$ X_k = \left[  \quotient{E_k}{\mu_k} \right]. $$

From now on, we fix $k$ and denote $X\coloneqq X_k$, $E\coloneqq E_k$, and $\mu\coloneqq \mu_k$.
%Moreover, $X_r$ is rigid in moduli for $r>2$. These three cases are the main object of this section.
Let  $Y\coloneqq \Tot(\omega_X)=\left[ \Tot(\omega_E)/\mu \right]$ be the total space of the cotangent orbifold bundle of $X$. We have a commutative diagram
\begin{center}
\begin{tikzcd}
E\rar\dar & \Tot(\omega_E) \dar \\
X \rar{\iota} & Y
\end{tikzcd}
\end{center}
where the vertical arrows are quotients by $\mu$ and the horizontal ones are inclusions via the zero section. 

Recall that a triangulated category $\mathbb{T}$ is called a \emph{K3-category} if the functor $[2]$ is a Serre functor, i.e. if for any two objects $E,F\in \mathbb{T}$ there is a natural isomorphism
$$ \Hom^\bullet(E,F) \xrightarrow{\sim} \Hom^\bullet (F,E[2])^*.  $$

Let $\sD$ denote the full triangulated subcategory of coherent sheaves supported on the zero section of $Y$. Then we have: 
\begin{lemma}\label{lem_sD_CY-category}
$\sD$ is a K3-category. In particular, the Euler form is symmetric. Moreover, for any $E,F\in D^b(X)$, one has 
$$ \Hom^\bullet_{\sD}(\iota_*E,\iota_*F)= \Hom^\bullet_{X}(E,F) \oplus \Hom^\bullet_{X}(F,E)^*[-2]. $$
In particular, $\chi_{\sD}(\iota_*E,\iota_*F)=\chi_{X}(E,F) + \chi_{X}(F,E).$
\end{lemma}

\begin{proof}
This follows from \cite[Lemma 4.4]{Kel11}.
\end{proof}

\begin{lemma}\label{lem_iota_iso_of_K_groups}
The map $\iota$ induces an isomorphism of abelian groups $K(X)\simeq K(\sD)$.
\end{lemma}

\begin{proof}
Let $X_n$ be the $n$-th order neighborhood of $X$ in $Y$. Denote by $\sB$ be the abelian category of sheaves supported on $X$. Then any $F\in \sB$ is an $\sO_{X_n}$-module for some $n$. Therefore, $F$ is obtained as a successive extension of $\sO_X$-modules, and the map
$$ \iota_*\colon K(X) \to K(\sB)= K(\sD) $$
is surjective. Let $\pi\colon Y\to X$ denote the projection to the zero section. Since $R^i\pi_*=0$ for $i>0$, the functor
$$\pi_*\colon \sB \to \Coh(X)$$
is exact. The induced map on $K$-groups is the inverse of $\iota_*$.
\end{proof}

\subsection{Exceptional and spherical objects}

An object $S\in \sD$ is called \emph{spherical} if $\Hom^\bullet(S,S)\simeq \C\oplus \C[-2]$. Suppose $S\in\sD$ is a spherical object. Given an object $G\in \sD$ we define $\Phi_S (G)$ to be the cone of the evaluation morphism
$$
\Hom^\bullet(S,G)\otimes S \xrightarrow{ev} G \to \Phi_S (G).
$$
Similarly, $\Phi^-_S (G)$ is a shift of the cone of the coevaluation map
$$ \Phi^-_S(G) \to G\xrightarrow{ev^*} \Hom^\bullet(G,S)^*\otimes S $$
The operations $\Phi_S$, $\Phi^-_S$ define autoequivalences of $\sD$, called \emph{spherical twists} \cite{ST01}.

Spherical twists act on $K(\sD)$ via reflections: if $S$ is a spherical object, and $[G]\in K(\sD)$, we have
\begin{equation}\label{eq_def_of_wS}
w_S([G])\coloneqq[\phi_S(G)]=[G]-\chi(S,G)[S]. 
\end{equation}

\begin{lemma}\label{lem_properties_of_twists}
Let $S$ be a spherical object of $\sD$. Then,
\begin{enumerate}[(i)]
    \item $\Phi_S\Phi_S^{-}\simeq \id_\sD$ and $\Phi_S^{-}\Phi_S\simeq \id_\sD$;
    \item $\Phi_S (S) \simeq S[-1]$;
    \item for any spherical object $S'$ such that $\Hom^\bullet(S',S)\simeq \C[-1]$, there is an isomorphism
    $$ \Phi_S\Phi_{S'}(S)\simeq S'. $$
\end{enumerate}
\end{lemma}

\begin{proof}
These properties follow from Proposition 2.10, Lemma 2.11 and Proposition 2.13 in \cite{ST01}.
\end{proof}

Next, we construct spherical objects (and  autoequivalences) of $\sD$. We do so starting from an exceptional collection of $D^b(X)$:

\begin{definition} Let $\mathbb T$ be a triangulated category. An object $E\in \mathbb T$ is \emph{exceptional} if $$\Hom^\bullet(E,E)=\C[0].$$ % and $\Hom(E,E[l])=0$ for $l\neq 0$. 
An \emph{exceptional collection} is a sequence of exceptional objects $E_1,...,E_n$ such that $\Hom^\bullet(E_i,E_j)=0$ for $i>j$.
We say that an exceptional collection is \emph{full} if it generates $\mathbb T$, i.e. $\mathbb T$ is the smallest triangulated category containing $\set{E_1,..,E_n}$.
\end{definition}

\begin{comment}
Suppose $E\in\sD$ is exceptional. Given an object $X\in E^\perp$ its \emph{right mutation across} $E$ is the object $\mathcal R_E(X)\in ^\perp\!E$ determined by the triangle
\begin{equation}\label{eq_tr_right_mutation}
\mathcal R_E(X)\to X\to \Hom^\bullet(X,E)^*\otimes E
\end{equation}
\end{comment}

%\AAA{notation clash: $i$ goes from 1 to $r$ in sec. 3, here $r$ indexes the quotient. Maybe use $k$?}

The category $\Coh(X)$ admits exceptional simple sheaves (see, for example, \cite{GL87}), described as follows. Identify $\Coh(X)$ with the category of $\mu$-equivariant sheaves on $E$, and denote by $p_i\in E$ the points with non-trivial stabilizer $\mu_{a_i}$. Let $\chi^0,...,\chi^{a_i-1}$ be the irreducible representations of $\mu_{a_i}$. The equivariant skyscraper sheaves $\sO_{p_i}\otimes \chi^{j}$ (with $j\in\{ 0,...,a_i-1 \}$) are exceptional objects of $\Coh(X)$.

Moreover, $D^b(X)$ admits several full exceptional collections \cite{Mel95}. We will use the following one:
\begin{equation*}
\begin{aligned}
\mathbb F\coloneqq && ( \sO_{p_1}\otimes\chi^{a_1-1} ,..., \sO_{p_1}\otimes\chi^{1}, \\
 && \sO_{p_2}\otimes\chi^{a_2-1}, ..., \sO_{p_2}\otimes\chi^{1}, \\
 && ..., \\
 &&\sO_{p_r}\otimes\chi^{a_r-1},..., \sO_{p_r}\otimes\chi^{1},\\
 && \sO ,\sO(1)).
\end{aligned}
\end{equation*}
%$$\mathbb F\coloneqq \left( \sO_{p_1}\chi^{r_1-1},..., \sO_{p_1}\chi^{1}, \sO_{p_2}\chi^{r_2-1}, ...,  \sO_{p_2}\chi^{1}, \sO_{p_3}\chi^{r_3-1},...,\sO_{p_3}\chi^{1}, \sO ,\sO(1) \right). $$

Exceptional objects in $\Coh(X)$ give rise to spherical objects in $\sD$:

\begin{proposition}\label{prop_push_exceptional_becomes_spherical} Suppose $E\in D^b(X)$ is exceptional, then $\iota_*E$ is a sperical object in $\sD$.
\end{proposition}
\begin{proof}
This is Proposition 3.15 in \cite{ST01}.
\end{proof}

By Prop. \ref{prop_push_exceptional_becomes_spherical}, pushing forward the objects of $\mathbb F$, we obtain a set of spherical objects:
\begin{equation}
    \label{eq_def_of_Pi}
\Pi\coloneqq \left\lbrace t_1^{a_1-1},..., t_1^1, t_2^{a_2-1}, ...,  t_2^1, ..., t_r^{a_r-1},...,t_r^1, \iota_*\sO ,\iota_*\sO(1)\right\rbrace, 
\end{equation}
where $t_i^j\coloneqq\iota_*(\sO_{p_i}\otimes\chi^j)$. 
We define the subgroup of $\Aut(\sD)$ generated by spherical twists across objects of $\Pi$:
\[ \Br(\sD)\coloneqq  \left\langle \Phi_S \in \Aut(\sD)\mid S\in \Pi \right\rangle\]

\begin{comment}
full exceptional collection of line bundles
$$ \mathbb E\coloneqq\left(\sO,\sO(x_1),...,\sO((r_1-1)x_1), \sO(x_2),...,\sO((r_2-1)x_2), \sO(x_3),...,\sO((r_3-1)x_1), \sO(1)\right). $$
The dual exceptional collection is
$$ \mathbb F\coloneqq \left(  \sO(1+\oo{})[1], t_1^{r_1-1},..., t_1^1, t_2^{r_2-1}, ...,  t_2^1, t_3^{r_3-1},...,t_3^1, \sO  \right)  $$
but we will more often use the slightly different collection:
$$\mathbb F'\coloneqq \left(  t_1^{r_1-1},..., t_1^1, t_2^{r_2-1}, ...,  t_2^1, t_3^{r_3-1},...,t_3^1, \sO ,\sO(1) \right). $$
\end{comment}

\subsection{The root system associated to \texorpdfstring{$\sD$}{mathcal{D}}}\label{ssec_root_system_associated_to_D}

In this section we use the spherical objects in $\Pi$ to construct an elliptic root system associated with $(K(\sD)_\R,\chi)$.

\begin{proposition}\label{prop_KD_is_a_root_system}
The set $R\coloneqq \set{[\Phi(S)]\in K(\sD) \st S\in\Pi, \Phi\in\Br(\sD)}$ satisfies the axioms of an extended root system associated to $\left(K(\sD)_\R,\chi_{\sD}\right)$ (see Def. \ref{def_gen_root_syst}). Moreover:
\begin{enumerate}[(i)]
% \item The radical of $\chi_\sD$ is generated by the image of $K(E)$ in $K(\sD)$ under the push forward along the quotient map $p\colon E\to X$;
        \item Define classes
        \[ a\coloneqq -[\sO_q] \ \ \mbox{ and } \ \  b\coloneqq[\iota_*(\sO_X\oplus\omega_X\oplus\omega_X^2)]. \]
    %$$a\coloneqq [\iota_*\sO_X]-[\iota_*\sO_X(1)],$$                      $$b\coloneqq[p_*\sO_{E}],$$
%where $p\colon E\to X$ is the quotient map.                      
Then $(a,b)$ is a basis of $\rad I$ and $a$ is a marking for $R$;
    \item\label{itm3_prop_KD_is_a_root_system} The Weyl group $W$ is generated by $\set{w_S \st S\in \Pi}$ (defined in \eqref{eq_def_of_wS});
    \item\label{itm4_prop_KD_is_a_root_system} the root systems arising from an elliptic orbifold quotient are precisely the ones described in Example \ref{ex_elliptic_EARS}. The vertices $(0,0),(0,1)$ correspond to $[\iota_*\sO_X(1)] ,[\iota_*\sO_X]$ respectively, and $(i,j)$ to $[t_i^j]$ (for $i\neq 0$).
\end{enumerate}
\end{proposition}

\begin{proof}
The axioms of an elliptic root system for $\left(K(\sD)_\R,\chi_{\sD}\right)$ are verified in \cite{Mel95}. Observe that the radical $\rad I$ has rank 2, and the classes $a,b$
%$$a= -[\sO_q] \ \ \mbox{ and } \ \  b=[\sO_X\oplus\omega_X\oplus\omega_X^2]$$
are invariant under twists by $\omega_{X}$, so $a,b\in\rad I$ by Lemma \ref{lem_omega_inv_radI} below.
\end{proof}

\begin{lemma}\label{lem_omega_inv_radI}
If $N\in D^b(X)$ satisfies $N\otimes\omega_X^*\simeq N$, then $[\iota_*N]\in \rad \chi_\sD$.
\end{lemma}

\begin{proof}
The classes $[\iota_*E]$ for $E\in D^b(X)$ generate $K(\sD)$, and we have
$$\chi_\sD(\iota_*N,\iota_*E)=\chi_X(N,E)+\chi_X(E,N)=\chi_X(N,E)-\chi_X(N\otimes \omega_{X}^*,E)=0$$
by Lemma \ref{lem_sD_CY-category}.
\end{proof}

In analogy with Notation \ref{not_vertices_index}, and in virtue of Prop. \ref{prop_KD_is_a_root_system}\eqref{itm4_prop_KD_is_a_root_system}, we write 
\begin{equation}\label{eq_notation_Si}
\begin{split}
S_{(0,1)} &\coloneqq \iota_*\sO_X\\
S_{(0,0)} &\coloneqq \iota_*\sO_X(1)\\
S_{(i,j)}&\coloneqq t_i^j \mbox{ for }i=1,...,r \mbox{ and }j=1,...,a_i-1
\end{split}
\end{equation} 
for the objects of $\Pi$. 

Let $\Gamma$ denote the diagram corresponding to $R$, and recall that the definitions of the underlying affine and finite Dynkin diagrams $\Gamma_a$ and $\Gamma_f$ (see Sec. \ref{ssec_DynkinGraph}). In analogy with Definition \ref{def_r_ij}, we introduce the following elements of $\Br(\sD)$:

\begin{definition}\label{def_rho_ij}
For each vertex of $\Gamma_a$ define elements of $\Br(\sD)$ inductively as follows:
\begin{enumerate}
    \item $\rho_{(0,1)} \coloneqq \Phi_{S_{(0,1)}}\Phi_{S_{(0,0)}}$;
    \item $\rho_{(i,1)}\coloneqq \Phi_{(t_i^1)}\rho_{(0,1)}\Phi_{(t_i^1)}\rho_{(0,1)}^{-1}$ for $i=1,...,r$;
    \item $\rho_{(i,j)}\coloneqq \Phi_{(t_i^j)}\rho_{(i,j-1)}\Phi_{(t_i^j)}\rho_{(i,j-1)}^{-1}$ for $i=1,...,r$, $j=2,...,a_i-1$;
\end{enumerate}
\end{definition}

\begin{comment}
\begin{definition}\label{def_r_ij}
For each vertex of $\Gamma_a$ define elements of $W$:
\begin{enumerate}
    \item $r_{0} \coloneqq w_{\alpha_0}w_{\alpha_{-1}}$;
    \item $r_{(i,1)}\coloneqq w_{\alpha_{(i,1)}} r_{0} w_{\alpha_{(i,1)}} r_{0}^{-1}$ for $i=1,...,r$;
    \item $ r_{(i,j)}\coloneqq w_{\alpha_{(i,j)}} r_{(i,j-1)} w_{\alpha_{(i,j)}} r_{(i,j-1)}^{-1}$ for $i=1,...,r$, $j=2,...,r_i-1$;
\end{enumerate}
\end{definition}
\end{comment}

By Prop. \ref{prop_KD_is_a_root_system}\eqref{itm3_prop_KD_is_a_root_system}, the assignment $\Phi_S \mapsto w_S$ defines a surjective homomorphism
$$ q\colon \Br(\sD) \twoheadrightarrow W.$$
It follows from the definitions and from the fact that $q$ is a homomorphism that $q$ maps the elements $\rho_{v}$ to the elements $r_{v}\in T < W$ for all $v\in \abs{\Gamma_a}$.

\subsection{Perverse sheaves and a heart in \texorpdfstring{$\sD$}{mathcal{D}}}\label{ssec_PerverseSheaves}

In this section we construct the heart of a bounded t-structure of $\sD$, denoted $\sA_R$, associated with the root system $R$. To do so, we consider the minimal resolution $Y'$ of $\overline{Y}$, the coarse moduli variety of the orbifold $Y = \Tot(\omega_X)$, and use the McKay correspondence \cite{BKR01}.

As a variety, $\overline{Y}$ has singularities of type $A_{a_i}$ at $p_i$. Then, the minimal resolution is $f : Y' \to \overline{Y}$, with $\mathbf{R} f_*\sO_{Y'}=\sO_{\overline{Y}}$ and exceptional locus the union of a chain of rational curves 
$$C_i \coloneqq \bigcup_{j=1}^{a_i-1} C_{i,j} \qquad j=1,..,r_i-1$$ above every point $p_i$. We write $X'\coloneqq X\cup (\cup_{i,j}C_{i,j})$ for the union of the exceptional curves with the strict transform of $X$. 

The derived McKay correspondence of \cite{BKR01} states that there is an equivalence 
$$\Psi\colon D(Y')\to D(Y), $$
which in turn induces an equivalence between $\sD$ and the full triangulated subcategory $\sD'$ of sheaves supported on $X'$. 
More precisely, $Y'$ can be realized as a moduli space of sheaves of $Y$ as follows. 

\begin{definition}
A $\mu$-equivariant quotient sheaf $\sO_{\Tot(\omega_E)}\twoheadrightarrow F$ is a $\mu$-\emph{cluster} if $H^0(F)$ is isomorphic to the regular representation of $\mu$ as a $\C[\mu]$-module. We regard $F$ as an element of $\Coh(Y)$.
\end{definition}

Let $\mu$-$\Hilb(Y)$ be the scheme parameterizing $\mu$-clusters on $Y$. Then, $\mu$-$\Hilb(Y)$ is a crepant resolution of $Y$ \cite{BKR01}, and the equivalence $\Psi$ is the Fourier-Mukai transform with kernel the universal family on $\mu$-$\Hilb(Y)\times Y$. Therefore we may pick $Y'\coloneqq \mu$-$\Hilb(Y)$  .

The inverse image of $\Coh(Y)$ under $\Psi$ is the abelian category of \textit{perverse sheaves} on $Y'$, which is obtained from $\Coh(Y')$ with the tilt below (we follow the notation of \cite{Bri02_flops} and \cite{VdB04_flops}). 
Let $\sC$ be the abelian subcategory of $D(Y')$ consisting of sheaves $E$ such that $\mathbf{R}f_*E=0$, and define a torsion pair:
\begin{equation}\label{eq_pair_perverse_surface}
\begin{split}
\sT'_0&\coloneqq \left\lbrace T\in \Coh(Y') \st \mathbf{R}^1f_*T=0  \right\rbrace  \\
\sF'_0&\coloneqq \left\lbrace F\in \Coh(Y') \st f_*F=0  \mbox{ and }\Hom(\sC,
F)=0  \right\rbrace 
\end{split}
\end{equation} 

We denote by $\Per(Y')$ the tilt of $
\Coh(Y')$ along the pair \eqref{eq_pair_perverse_surface}, i.e. $ \Per(Y')\coloneqq \pair{\sF'_0[1],\sT'_0}$. This results in a diagram whose horizontal arrows are equivalences:
\begin{center}
\begin{tikzcd}
\Per(Y') & \Coh(Y) \\
	\Coh(Y') & \Psi(\Coh(Y'))
	\arrow["\Psi", from=1-1, to=1-2]
	\arrow["\Psi", from=2-1, to=2-2]
	\arrow["tilt"', leftrightsquigarrow, from=1-1, to=2-1]
	\arrow["tilt", leftrightsquigarrow, from=1-2, to=2-2]
\end{tikzcd}
\end{center}

Denote by $\sB$ and $\sB'$ the intersections of $\Coh(Y)$ and $\Coh(Y')$, respectively, with $\sD$ and $\sD'$. Observe that $(\sT'_0\cap \sD',\sF'_0\cap \sD')$ is a torsion pair of $\sB'$: we denote by $\Per(X')$ the corresponding tilt. Define $\mathcal{A}_R\coloneqq \Psi(\sB')$. Then, restricting the above diagram to $\sD'$ and $\sD$ yields: 
\begin{center}
\begin{tikzcd}
\Per(X') & \sB \\
	\sB' & \mathcal{A}_R
	\arrow["\Psi", from=1-1, to=1-2]
	\arrow["\Psi", from=2-1, to=2-2]
	\arrow["tilt"', leftrightsquigarrow, from=1-1, to=2-1]
	\arrow["tilt", leftrightsquigarrow, from=1-2, to=2-2]
\end{tikzcd}
\end{center}

In particular, the equivalence $\Psi$ maps the simple objects of $\Per(X')$ into simple sheaves in of $\sB$:
\begin{equation*}
    \begin{split}
        \sO_{C_{i,j}}(-1) \longmapsto t_i^j;\\
        \sO_{C_i}(C_i)[1] \longmapsto t_i^0 
    \end{split}
\end{equation*}
and moreover $\sO_{X'} \longmapsto \sO_X$ \cite[Sec. 3.5]{VdB04_flops}.

\begin{remark}
The category $\Per(Y')$ is usually called the category of 0-perverse sheaves. Its dual category of (-1)-perverse sheaves is used in \cite{Bri02_flops} and \cite{Tod13_extremal}, and the two are compared in \cite[Sec. 3.5]{VdB04_flops}. Our choice of $\Psi$, and therefore of the perversity of $\Per(Y')$, has the advantage of mapping skyscraper sheaves to clusters.
\end{remark}

\begin{lemma}\label{lem_A_Noeth}
$\mathcal{A}_R$ is Noetherian. 
\end{lemma}
\begin{proof}
This is straightforward, because $\sB'$ is Noetherian.
\end{proof}

%\AAA{this has no label, where is it used?}
%\begin{lemma}
%Clusters in $\mathcal{A}_R$ are simple objects of class $a$.%
%\end{lemma}
%\begin{proof}
%If $M$ is a cluster contained in $\mathcal{A}_R$, then $M=\Psi(\sO_t)$ for some $t\in X'$, by definition of $\sB'$. Skyscraper sheaves are simple in $\sB'$, and so is $M$ in $\mathcal{A}_R$. Since free orbits are clusters, all clusters have class $a=[\C_p]$ in $K(\sD)$.
%\end{proof}

To classify objects of $\mathcal{A}_R$ we will describe it explicitly as a tilt of $\sB$. 
Define $\sF'$ to be the full additive subcategory of $\sB'$ generated as the extension closure of subsheaves of the normal bundles $\sO_{C_i}(C_i)$:
$$\sF'=\pair{F \st F \subseteq \sO_{C_{i}}(C_i)\in\sB' \text{ for }i=1,...,r}$$
and $\sT'$ to be its left orthogonal in $\sB'$.
Denote by $\sF$ (resp. $\sT$) the subcategories $\Psi(\sF')$ (resp. $\Psi(\sT')$) of $\mathcal{A}_R$.

\begin{lemma}\label{lem_(T'F')_tp}
$(\sT',\sF')$ is a torsion pair in $\sB'$. %Therefore, the pair $(\sT,\sF)$ is a torsion pair in $\mathcal{A}_R$ and  $\pair{\sF[1],\sT}=\sB$.
\end{lemma}

\begin{proof} We follow an argument similar to \cite[Lemma 3.2]{Tra17}.
We need to show that every sheaf $E\in \sB'$ fits in a short exact sequence 
$$T\to E\to F$$
with $T\in\sT'$, $F\in\sF'$. If $E\in\sT'$, we are done. Otherwise, $\Hom(E,\sF)\neq 0$, so there exists $F_1\in\sF'$ fitting in a short exact sequence 
$$M_1\to E\to F_1.$$
If $\Hom(M_1,\sF')\neq 0$, repeat this process, and obtain
$$M_2\to M_1\to F_2.$$
By iterating this, we get a chain of inclusions
$$ ...\subset M_k \subset M_{k-1} \subset ... \subset M_1 \subset E $$
with quotients in $\sF'$. Then, the chain must terminate by Lemma \ref{lem_inclusion-chain-terminate}. This means that there exists $n$ for which $\Hom(M_n,\sF')=0$. Let $F$ be the cokernel of the inclusion $M_n\subset E$, then the sequence
$$M_n\to E\to F$$
is the desired one.
%The fact that $\Psi$ is an equivalence implies the statement about $\mathcal{A}_R$. By construction, all objects in $\pair{\sF[1],\sT}$ are sheaves, so we can apply Lemma \ref{lem_HeartsContainEachOther} and conclude $\pair{\sF[1],\sT}=\sB$. 
\end{proof}

\begin{lemma}[{see \cite[Lemma 3.1]{Tra17}}]\label{lem_inclusion-chain-terminate}
If there is a series of inclusions in $\sB'$, say
$$ ...\subset M_k \subset M_{k-1} \subset ... \subset M_0 $$
whose quotients lie in $\sF'$, then the sequence must eventually stabilize.
\end{lemma}

\begin{proof}
We follow an argument similar to \cite[Lemma 3.1]{Tra17}.
First, we may assume that all the quotients $F_k\coloneqq M_0/M_k$ are supported on one curve $C\coloneqq C_i$. Moreover:

\begin{claim*}
We may assume that for all $k$, the quotients $F_k$ are torsion free sheaves $L_k\subset\sO_C{(C)}$, such that $L_k$ has connected support $D_k\subset C$.
\end{claim*}

Indeed, by definition of $\sF'$ every $F_k$ admits a surjection to some $L_k\subset \sO_{C}(C)$. By restricting $L_k$ to one of the connected components $D_k$ of its support, we may assume that $L_k$ has connected support. 
So we have quotients 
$$F_k\twoheadrightarrow  L_k$$
which define exact sequences 
$$0\to M_k^{(1)}\to M_k \to  L_k\to 0.$$
The quotient $F_k^{(1)}$ of $M_{k+1}\to M_k^{(1)}$ fits into an exact sequence 
$$ F_k^{(1)} \to F_k \to L_k $$
where $\ch{1}{(F_k^{(1)})}=\ch 1{(F_k)}- \ch 1{(L_k)} $ is a positive linear combination $\sum a_j [C_{i,j}]$ with coefficients strictly smaller than those of $\ch1{(F_k)}$. We can then repeat this process for the map $M_{k+1}\to M_k^{(1)}$ until we get a finite chain of inclusions 
$$M_{k+1}\subset M_k^{(n)} \subset ... \subset M_k^{(1)} \subset M_k$$
satisfying the statement of the claim.

%Our goal is now to show that if $L_{i,E'}\neq 0$, then $\Hom(M_i,\sO_C(C))>\Hom(M_{i+1},\sO_C(C))$. This implies that the chain of subobjects must terminate. 

We proceed to show that the sequence of inclusions must terminate with an induction on the length $l$ of the chain of rational curves $C$. 

In order to see this, apply the functor $\Hom(-,\sO_C(C))$ to the short exact sequence 
\begin{equation}\label{eq_Mi1_Mi_Li}
    0\to M_{k+1}\to M_k \to  L_k\to 0.
\end{equation}

For $L_k = \sO_C(C)$, one computes $\ext^1(\sO_C(C),\sO_C(C))=0$, hence $\Hom(M_i,\sO_C(C))>\Hom(M_{i+1},\sO_C(C))$.

If $L_k\subsetneq \sO_C(C)$, one has \begin{equation}\label{eq_ext2_vanishes}
    \Ext^2(L_k,\sO_C(C))\simeq \Hom (\sO_C(C),L_k)=0,
\end{equation}  and obtains
\begin{equation}
    \begin{split}
        \hom(M_k,\sO_C(C)) - \hom(M_{k+1},\sO_C(C)) = \\ \chi(L_k,\sO_C(C)) + \left(\ext^1(M_k,\sO_C(C))-\ext^1(M_{k+1},\sO_C(C))\right).
    \end{split}
\end{equation}
Observe that $\chi(L_k,\sO_C(C)) = - (D_k). C\geq 0$ by Hirzebruch-Riemann-Roch, and that $$\ext^1(M_k,\sO_C(C))-\ext^1(M_{k+1},\sO_C(C))\geq 0$$
because of \eqref{eq_ext2_vanishes}.

If $l=1$, we must have $D_k=C$ and $-D_k \cdot C=2$. This shows that if $L_k\neq 0$, then $\Hom(M_k,\sO_C(C))>\Hom(M_{k+1},\sO_C(C))$, whence the chain of subobjects must terminate. 

If $l>1$, the only way the sequence does not terminate is that all $L_k$ satisfy $D_k\cdot C=0$. This is only possible if no $D_k$ contains the terminal curves of the chain, $C_1$ and $C_l$, in their support. In other words, $L_k\subset \sO_{C}(C)_{|C'}\simeq \sO_{C'}(C')$ where $C'=\cup_{j=2}^{l-1}C_j$ is a shorter chain. Then, we can repeat the argument above applying the functor $\Hom(-,\sO_{C'}(C'))$ to the sequences \eqref{eq_Mi1_Mi_Li}. Eventually, the problem is reduced to the case $l=1$, and the process must terminate.
\end{proof}

\begin{proposition}
We have $\sF'=\sF'_0$. Therefore, 
\[ \mathcal{A}_R= \pair{\sT,\sF} \quad \mbox{ and } \quad \sB=\pair{\sF[1],\sT}. \]
\end{proposition}

\begin{proof}
Suppose $E\in \sF'$. We may assume that $E$ is supported on just one curve $C=C_i$. Moreover, $E$ is a repeated extension of subsheaves of $\sO_{C}(C)$, so we may induce on the number of its factors and reduce to the case where $E$ is a subsheaf of $\sO_C(C)$. 

It follows from left exactness of $f_*$ that $f_*E=0$. Now suppose $U\in \sC$. Composing a map $U\to E$  with the inclusion $E\subset \sO_C(C)$ yields an element of $\Hom_{Y'}(U,\sO_C(C))$. $U$ must be supported on $C$ since $f\colon Y' \to Y$ is an isomorphism off $C$. Therefore we have isomorphisms 
\begin{equation}
\label{eq_gimmick}
\Hom_{Y'}(U,\sO_C(C))\simeq \Hom_C(U,\sO_C(C)) \simeq \Ext^1_C(\sO_C,U)^* = H^1(C,U)^*=0 
\end{equation}
since $U\in \sC$. We conclude $\Hom(U,E)=0$ for all $U\in \sC$, and $E\in \sF_0'$.

% Recall that $\sC$ is the additive category generated by $\{\sO_{C_i,j}(-1)\}$.  Then, any non-zero map $\sO_{C_{i,j}}(-1)\to E$ induces a non-zero map $\sO_{C_{i,j}}(-1)\to\sO_C(C)$. However, it is straightforward to see that $\Hom(\sO_{C_{i,j}}(-1),\sO_C(C))=0$. Indeed the maximal line bundle $L_j$ on $C_{i,j}$ which maps to $\sO_C(C)$ is the kernel of the restriction of $\sO_C(C)$ to $\overline{C\setminus C_{i,j}}$, which has degree $-2$. This shows $\sF'\subseteq \sF'_0$.

Conversely, suppose $E\in\sF'_0$, and assume $E\notin \sF'$. Then there is a short exact sequence 
\[ G\to E\to H \]
with $H\in \sF'$ and $G\in \sT'$ since $(\sT',\sF')$ is a torsion pair by Lemma \ref{lem_(T'F')_tp}. Moreover, $G\in \sF_0'$ because $E\in \sF_0'$. Then $G\in \sT'\cap \sF_0'$ satisfies
\begin{itemize}
\item $\Hom(G, \sF')=0$;
\item $f_*G=0$;
\item $\Hom(\sC,G)=0$.
\end{itemize}
We must have $\mathbf{R}^1f_*G\neq 0$, otherwise $\mathbf{R}f_*G=0$ and $G\in \sC$. Therefore, $H^1(C,G)\neq 0$. Arguing as in \eqref{eq_gimmick} we get $\Hom_{Y'}(G,\sO_C(C))\neq 0$, contradicting $\Hom(G,\sF')=0$. We conclude that $G=0$, and therefore $E\in \sF'$.
\end{proof}

\subsection{Classification of objects in \texorpdfstring{$\mathcal{A}_R$}{the heart}}\label{ssec_ClassificationA}
Next, we classify of objects in $\mathcal{A}_R$. The strategy is to explicitly compute the image under $\Psi$ of objects of $\sB'$, which are obtained as finite, repeated extensions of torsion sheaves and line bundles on each component of $X'$. We start describing sheaves in $\sF'$: these are precisely all elements of $\sB'$ whose image is a shift of a sheaf in $\mathcal{A}_R$. 

Given a subchain of rational curves $D \subseteq C$, there exists a maximal subsheaf $L_D\subseteq \sO_{C}(C)$ supported on $D$.

\begin{lemma}\label{lem_description of LD}
Fix $C=C_i$, let $D\subseteq C$ be a subchain of rational curves, and let $L_D$ as above. Write $C_{d_1},..., C_{d_l}$ for the irreducible components of $D$ (with $(d_1,...,d_l)$ consecutive elements of $\set{1,...,r_i-1}$). Then $L_D$ is obtained from $\sO_{C_{d_1}}(-2)$ with repeated extensions by the sheaves $\sO_{C_{d_i}}(-1)$, with $i=d_2,...,d_l$. 
In particular, there is a short exact sequence 
\begin{equation}
    \label{eq_LD_LD'_Ot}
 L_D \to L'_D \to \sO_t
\end{equation}
where $t\in C_{d_1}$ and $L'_D$ is obtained by repeated extensions of $\sO_{C_{d_i}}(-1)$, with $i=d_1,...,d_l$.
\end{lemma}

\begin{proof}
Proceed by induction on the length $l$ of the chain $D$. If $l=1$ and $D=C_d$, one readily verifies that $L_D\simeq\sO_{C_d}(-2)$. Suppose then that $l>1$. Then, observe that $L_{D}$ restricts to $C_{d_l}$ to a line bundle of degree $-1$, because either $d_l<r_i-1$, and then sections of $L_D$ must vanish at the intersection $C_{d_l}\cap C_{d_l+1}$ or because $d_l=r_i-1$, and $\sO_C(C)$ has degree $-1$ on $C_{r_i-1}$. The kernel of this restriction is exactly the maximal subsheaf of $\sO_C(C)$ supported on $\overline{D-C_{d_l}}$. In other words, $L_D$ fits in a short exact sequence
$$ L_{\overline{D-C_{d_l}}} \to L_D \to \sO_{C_{d_l}}(-1) $$
so by induction $L_D$ has the asserted structure. 

For the second statement, fix a point $t\in C_{d_1}$ away from the intersections, and consider the cokernel
$$ (\epsilon)\colon \quad \sO_{C_{d_1}}(-2) \to L_D \to R_D. $$
From the sequence 
$$ \sO_{C_{d_1}}(-2) \to \sO_{C_{d_1}}(-1) \to \sO_t $$
one sees that $\Ext^1(R_D, \sO_{C_{d_1}}(-2))\simeq \Ext^1(R_D, \sO_{C_{d_1}}(-1))$ because $t\notin \text{Supp} R_D$. Pushing forward the extension class $(\epsilon)$ to $\Ext^1(R_D, \sO_{C_{d_1}}(-1))$ produces an object $L_D'$ as in the statement.
\end{proof}

\begin{comment}
We write $t_i^{lk}$ for the unique extension
$$ t_i^l \to t_i^{l,k}\to t_i^k $$
for all $p_i\in R$, $j,k=0,...,r_i-1$, and $l=k\pm 1 \text{ mod }r_i$.
\end{comment}

\begin{lemma}\label{lem_classification-torsion-in-A}
Suppose an object $T\in \mathcal{A}_R$ is supported on an orbifold point $p_i$. Then $T$ is obtained by repeated extensions of the following objects:
\begin{enumerate}[(i)]
    \item \label{itm_tij}$t_i^j$ with $j\neq 0$;
    \item \label{itm_clusters}clusters supported at $p_i$;
    \item \label{itm_N[-1]}$N[-1]$ where $N$ is the proper quotient of a cluster.
\end{enumerate}
\end{lemma}
\begin{proof}
This is equivalent to classifying sheaves of $\sB'$ supported on $C\coloneqq C_i$. First, we consider sheaves in $\sF'$. A sheaf in $\sF'$ is an extension of subsheaves $L\subset\sO_C(C)$ with connected support. Any such inclusion must factor thorugh an inclusion $L\subseteq L_D$, where $L_D$ is as in Lemma \ref{lem_description of LD} and the cokernel $L_D/L$ is torsion. We have that $\Psi(L)[1]$ and $\Psi(L_D)[1]$ are sheaves on $X$, so applying the McKay functor to 
$$ L\to L_D \to L_D/L $$
we obtain a short exact sequence of sheaves in $\sB$:
$$  M \to \Psi(L)[1]\to \Psi(L_D)[1], $$
where $M$ is obtained by repeated extensions of clusters. Now we claim that $\Psi(L_D)[1]$ is a proper quotient of a cluster. In fact, apply $\Psi$ to the exact sequence \eqref{eq_LD_LD'_Ot} of Lemma \ref{lem_description of LD}: $\Psi(\sO_t)$ is a cluster, and $\Psi(L'_D)$ is a sheaf obtained by repeated extensions of $t_i^j$, $j\neq 0$. This yields a short exact sequence in $\sB$
$$ 0\to \Psi(L'_D)\to \Psi(\sO_t) \to \Psi(L_D)[1] \to 0 $$
which exhibits $\Psi(L_D)[1]$ as the quotient of a cluster. This exhausts part \eqref{itm_N[-1]}.

Now, consider a sheaf $B\in\sT'$. The torsion part $B_{tor}$ of $B$ is obtained by repeated extensions of points, so $\Psi(B_{tor})$ is as in part \eqref{itm_clusters}. We may then assume that $B$ is torsion free with connected support. If $B$ is supported on a single irreducible component $C_i$, then $B$ is a sum of line bundles of the form $\sO_{C_i}(k)$. Since $\Hom(B,\sF')=0$, we must have $k>-2$. Then $\Psi(B)$ is obtained as an extension of $t_i^j$ by clusters. 
If $B$ is supported on more than one irreducible component, suppose that $C_j$ is a terminal component of the support of $B$ and consider the restriction of $B$ to $C_j$. Then there is an exact sequence
$$ B'\to B\to B_{|C_j} $$
where $B'$ is supported on a shorter chain. $B_{|C_j}$ is supported on one irreducible curve, so it is as above. If $B'\in\sT'$, we repeat this procedure. Otherwise, $B'$ fits in a short exact sequence of sheaves 
$$ B''\to B'\to F $$
with $B''\in\sT'$ and $F\in\sF'$. Sheaves in $\sF'$ are classified above, so we can assume that $B'\in\sT'$ and conclude by induction on the length of the supporting chain.
\begin{comment}
Observe that $\Psi\sO_{C_j}(-2)=N[-1]$ where $N$ fits in a sequence
$$ t_i^j\to M\to N $$
with $M$ a cluster.
Our goal is to show that all sheaves $F\in\sF$ fit in an exact sequence $$ F\to F'\to M $$ where $M$ is a successive extension of skyscraper sheaves, and $\Psi F'[1]$ is a proper quotient of a cluster. 

then its image $N\coloneqq \Psi F$
Suppose $L$ is a subsheaf of $\sO_{C}(C)$ with connected support $D$, and let $M_D$ be the maximal subsheaf of $\sO_C(C)$ supported on $D$. 
\end{comment}
\end{proof}

As a consequence of the results in this section, we obtain the following description of objects in $\mathcal{A}_R$:

\begin{proposition}\label{prop_classification-objects-A}
Objects in $\mathcal{A}_R$ are obtained by repeated extensions from:
\begin{enumerate}[(i)]
\item line bundles on $X$;
\item skyscraper sheaves $\sO_q$ for $q\in X-\cup \set{p_i}$;
\item objects supported on the $p_i$'s, classified in Lemma \ref{lem_classification-torsion-in-A}.
\end{enumerate}
\end{proposition}

\subsection{The fundamental region and normalization}\label{ssec_FundRegionU}

Recall the notation introduced in Section \ref{sec_ell_root_syst} and the identification $K(\sD)_\R\simeq F$. In this section, we use the heart $\sA_R$ to construct a region $U$ in $\Stab(\sD)$ which is a homeomorphic lift via $\pi\colon \Stab(\sD)\to \Hom(F,\C)$ of the fundamental domain $D$ described in Proposition \ref{prop_def_fundamental_domain}. Then, following \cite{Bri09_kleinian}, we introduce \emph{normalized} stability conditions.

\begin{proposition}\label{prop_FundamentalRegionU}
For every point $Z$ in the fundamental domain $D\subset \mathbb E$ there exists a unique stability condition $(Z,\mathcal{A}_R)\in \Stab(\sD)$. In fact, the inverse image $U\coloneqq \pi^{-1}(D)$ maps homeomorphically to $D$ under the central charge map.  \end{proposition}

\begin{proof}
Pick $Z\in D_\tau\subset D\subset \mathbb E$. The class of every object in $\mathcal{A}_R$ is a positive linear combination of classes of objects listed in Prop. \ref{prop_classification-objects-A}. Then, the definition of $D_\tau$ shows that $Z(\mathcal{A}_R)\subset \overline\HH$, in other words, $Z$ is a stability function on $\mathcal{A}_R$. Since $\mathcal{A}_R$ is Noetherian (Lemma \ref{lem_A_Noeth}), and the image of $\im Z$ is discrete by construction, then $Z$ has the Harder-Narasimhan property by Prop. \ref{prop_BriHNproperty}.

Again by Prop. \ref{prop_classification-objects-A}, we see that the image of $Z$ is discrete, so the support property is automatically satisfied.
Then, the map $\pi_{|U}$ is a homeomorphism.  
%Then, $(Z,\mathcal{A}_R)\in \Stab_n(\sD)$ since $Z\in \mathbb{E}$ implies $Z(a)=1$. 
\end{proof}

We observe right away the following Lemma:

\begin{lemma}\label{lem_Pi_are stable in U}
Let $\sigma\in U$. Then, all $t_i^j$, $j\neq 0$, and all line bundles $\sO_{X}(d)$ are $\sigma$-stable.
\end{lemma}

\begin{proof}
Let $S$ be one of the objects in $
Pi$ (see 
\eqref{eq_def_of_Pi}) or a sheaf $\sO_X(d)$. A short exact sequence 
\begin{equation}\label{eq_ses_Si_stable}
   K\to S\to Q 
\end{equation}
in $\mathcal{A}_R$ corresponds under the McKay functor to a short exact sequence of sheaves on the resolution 
$$ K'\to \Psi^{-1} (S) \to Q'. $$
On the other hand, $\Psi^{-1} S $ is either an object of the form $\sO_{C_{i,j}}(-1)$ or a line bundle on $X$. In either case, the only quotients of $\Psi^{-1} (S)$ are obtained by repeated extensions of skyscraper sheaves, so $Q\in \mathcal{A}_R$ is semistable of phase 1. Therefore $S$ is $\sigma$-stable.
\end{proof}

Let $\Stab^\dagger(\sD)$ be the connected component of $\Stab(\sD)$ containing $U$. In addition to the full stability manifold $\Stab(\sD)$, we will often restrict our attention to the locus of \textit{normalized} stability conditions 
\begin{equation}
\Stab_n(\sD)\coloneqq\{ \sigma=(Z,\PP) \in \Stab^\dagger(\sD) \st Z(a)=1 \}.
\label{eq_def_StabN}
\end{equation}

By construction, $U\subset \Stab_n(\sD)$, so we also define $\Stab_n^\dagger(\sD)\subset \Stab^\dagger(\sD)$ as the connected component of $\Stab_n(\sD)$ containing $U$. We use $\pi$ to denote the restriction of the central charge map to any of these regions of $\Stab(\sD)$. 
As it turns out, we have 

\begin{proposition}\label{prop_circ=dagger}
All stability conditions in $\Stab^\dagger(\sD)$ (and hence in $\Stab^\dagger_n(\sD)$) satisfy the additional condition 
\[ (*)\colon \quad  \im \frac{Z(b)}{Z(a)}>0.\]
\end{proposition}

The proof of Proposition \ref{prop_circ=dagger} uses our wall-crossing result (Theorem \ref{thm_enough_spherical_objects}) and is given in Section \ref{sec_proof_dagger=circ}.
An immediate consequence of Proposition \ref{prop_circ=dagger} is that $\pi$ maps $\Stab_n^\dagger(\sD)$ in $\mathbb{E}\subset \Hom(K(\sD),\C)$. This is used in Section \ref{sec_stabCondOnD}.

\begin{remark}\label{rmk_normalization}
Normalization is a very natural choice in this context: it already appears in the case of Kleinian singularities \cite{Bri09_kleinian} and it fits well with Saito's definitions of $\mathbb{E}$ and $\mathbb{H}$ (see \eqref{eq_def_EandH}), which include the condition $Z(a)=1$.

Moreover, as is the case in \cite{Tod08_crepant} and, for example, in \cite{IW19}, normalizing preserves information about the whole component $\Stab^\dagger(\sD)$. Indeed, $\Stab^\dagger(\sD)$ is the orbit of $\Stab^\dagger_n(\sD)$ under the $\C$-action, and it is a $\C^*$-bundle over the normalized locus $\Stab_n(\sD)$: these statements are proven in Section \ref{ssec:C_Orbit} using results from Section \ref{sec_WallCrossing}.  %allows to avoid the $\C$-action on $\Stab(\sD)$, as suggested by Toda in \cite{Tod08_crepant} and applied, for example, in \cite{IW19}. 
%Idea here: do 5.1 in general, with only stab (no normalization), until we get U. Then in 5.2, define Stab circ, Stab dagger, and the normalized versions. State 5.5 both for normalized and not. The rest of section 5 uses normalized, and section 6 does not. At the end of 6, bring in normalization (only change 6.3). 
\end{remark}

%--------------------------------------------------------------------------------------------------------------

\section{Wall-crossing in \texorpdfstring{$\sD$}{mathcal{D}}}\label{sec_WallCrossing}

In this section, we apply the wall-crossing methods of \cite{BM14_MMP} and \cite{BM_local_p2} to the $K3$-category $\sD$. First, we produce stable objects for a certain stability condition in $\Stab^\dagger(\sD)$. We then analyze wall crossing for spherical and radical classes, obtaining Theorem \ref{thm_enough_spherical_objects}. From it, we obtain a proof of Proposition \ref{prop_circ=dagger} and of the claims of Remark \ref{rmk_normalization}. The results of this section hold if one works with normalized stability conditions with the same arguments, so we do not repeat them. The notation is as above.

\subsection{Stability conditions on \texorpdfstring{$\Coh(X)$}{Coh(X)} and \texorpdfstring{$\sB$}{mathcal B}}

Geigle and Lenzing define slope-stability on a weighted projective line in \cite[Sec. 5]{GL87}. Define a stability condition $\tau_0'\coloneqq(Z_0,\Coh(X))\in\Stab(X)$ with 
$$ Z_0=-\deg + i\rk, $$
where $\deg(\sO_{p_i}\otimes\chi^j)$ is defined to be $\frac{1}{a_i}$ for all orbifold points $p_i$ and all $j=0,...,a_i-1$. Then, slope-stability is equivalent to $\tau_0'$-stability on $X$.  We say that a root $\alpha\in R\cup \Delta_{im}$ is \emph{positive} if $Z_0(\alpha)\in\mathbb H\cup \R_{<0}$. Results about $\tau_0'$-stability are summarized in \cite{LM93}:

\begin{thm}[{\cite[Theor. 4.6]{LM93}}]\label{thm_CB_indecomposables}
Let $X$ be as above, $\alpha\in R\cup \Delta_{im}$. Then:
\begin{enumerate}[(i)]
    \item there exists an indecomposable sheaf $F$ of class $\alpha$ if and only if $\alpha$ is a positive root;
    \item the sheaf $F$ is unique up to isomorphism if $\alpha$ is a real root, and varies in a one-parameter family if $\alpha$ is imaginary;
    \item an indecomposable sheaf is $\tau_0'$-semistable, and it is $\tau_0'$-stable if and only if $\alpha$ is primitive. 
\end{enumerate}
\end{thm}

\begin{comment}
\begin{proposition}[{\cite[Thm. 5.6]{GL87}}]\label{prop_GL_indecomposable_implies_semistable}
Every indecomposable vector bundle on $X$ is $\tau_0'$-semistable. 
\end{proposition}

Combining Prop. \ref{prop_GL_indecomposable_implies_semistable} with the following result, one produces $\tau_0'$-semistable objects in $\Coh(X)$. 

\begin{thm}[{\cite[Theor. 4.6]{LM93}}]\label{thm_CB_indecomposables}
\AAA{rewrite}
Let $\alpha\in R$ be a positive root. Then there exists an indecomposable sheaf in $\Coh(X)$ of class $\alpha$. There is a unique indecomposable for a real root, infinitely many for an imaginary one.
\end{thm}
\end{comment}

By Lemma \ref{lem_iota_iso_of_K_groups}, we can regard $Z_0$ as a map defined on $K(\sD)$, and define a stability condition $\tau_0\in\Stab(\sD)$ as $(Z_0,\sB)$. By construction, $\tau_0$ lies in the boundary of a fundamental chamber in $\Stab^\dagger(\sD)$ (for example because $\im Z_0(t_i^j)=0$ for all $i,j$). 

We say that an object $E\in\sD$ is \emph{semi-rigid} if $\ext^1(E,E)=2$. Then we have:

\begin{proposition}\label{prop_CB_on_local_setting}
Let $\alpha\in R\cup \Delta_{im}$ be a positive root. If $\alpha$ is a real root, there exist a $\tau_0$-semistable spherical sheaf in $\sB$ of class $\alpha$. If $\alpha$ is imaginary, there is a one-parameter family of semi-rigid $\tau_0$-semistable sheaves in $\sB$ of class $\alpha$. If $\alpha$ is primitive, the same statement holds with stability instead of semistability. 
\end{proposition}

\begin{proof}
By Theorem \ref{thm_CB_indecomposables}, there exists a $\tau_0'$-semistable sheaf $E'$ on $X$ of class $\alpha$. Let $E\coloneqq \iota_*(E')$ be the indecomposable sheaf in $\sB$ obtained by pushing forward $E'$. The sheaf $E$ is $\tau_0$-semistable: since $E$ is supported on $X$ then so must be every subsheaf $S\subset E$. This implies that $S=\iota_*S'$ for some $S'\in\Coh(X)$. Then, $S$ destabilizes $E$ if and only if $S'$ destabilizes $E'$. 

Next, we show that $E$ is spherical if $\alpha$ is a real root. Deformations of $E'$ are governed by the group $\Ext^1_X(E',E')$, so  Theorem \ref{thm_CB_indecomposables} implies that $\Ext^1_{X}(E',E')=0$, hence $\Ext^1_\sB(E,E)=0$ by Lemma \ref{lem_sD_CY-category}. On the other hand, since $\alpha$ is real one must have $\chi(\alpha,\alpha)=2$, so $E$ is spherical. Similarly, one argues that $E$ is semi-rigid if $\alpha$ is imaginary. The claim about stability follows again from Theorem \ref{thm_CB_indecomposables}.
\end{proof}

\subsection{Wall-crossing in \texorpdfstring{$\Stab(\sD)$}{Stab(D)}}

The lattice $K(\sD)$ can be equipped with the Mukai pairing  
$$ (\bfv,\bfw )\coloneqq -\chi(\bfv,\bfw). $$
The pairing has a rank 2 radical $\rad \chi$ generated by $a$ and $b$, and it induces a negative definite pairing on $K(\sD)/\rad\chi$, since the Euler form on $K(\sD)/\rad\chi$ coincides with the Cartan matrix of the root system $R_f$, which is positive definite. 

Since $K(\sD)$ is negative semidefinite, the class $\mathbf{v}$ of a stable object can only satisfy $\mathbf{v}^2=0$ or $\mathbf{v}^2=-2$. In the first case, $\mathbf{v}$ belongs to $\rad\chi$, and we call it a \textit{radical class}. Classes with $\mathbf{v}^2=-2$ are called \textit{spherical classes}.

First, notice that since $K(\sD)$ is a discrete lattice, we have a finiteness result for walls:

\begin{proposition}[{\cite[Prop. 3.3]{BM_local_p2}}]\label{prop_BM_finiteness_of_walls}
Let $\sD$ be a triangulated category such that $K(\sD)$ is a lattice of finite rank. Let $\Stab^*(\sD)\subset \Stab(\sD)$ be a connected component of its space of stability conditions. Fix a primitive class $\bfv\in K(\sD)$, and an arbitrary set $S \subset D$ of objects of class $\bfv$. Then there exists a collection of walls $W^S_\bfw$ with $\bfw\in K(\sD)$, with the following properties: 
\begin{enumerate}[(a)]
    \item Every wall $W^S_\bfw$ is a closed submanifold with boundary of real codimension one;
    \item The collection $W^S_\bfw$ is locally finite (i.e., every compact subset $K\subset \Stab^*(\sD)$ intersects only a finite number of walls);
    \item For every stability condition $(Z,\PP) \in W^S_\bfw$, there exists a phase $\phi$ and an inclusion $F_\bfw\to E_\bfv$ in $\PP(\phi)$ with $[F_\bfw] = \bfw$ and some $E_\bfv\in S$;
    \item If $\sC\subset \Stab^*(\sD)$ is a connected component of the complement of $\cup_{\bfw\in K(\sD)} W^S_\bfw$, and $\sigma_1,\sigma_2\in \sC$, then an object $E_\bfv\in S$ is $\sigma_1$-stable
if and only if it is $\sigma_2$-stable.
\end{enumerate}
\end{proposition}

Recall that $\sigma\in \Stab(\sD)$ is said to be \emph{generic} with respect to $\bfv\in K(\sD)$ if $\sigma$ does not lie on any of the walls of the wall-and-chamber decomposition associated to $\bfv$. The goal of this section is to prove the following Theorem:

\begin{thm}\label{thm_enough_spherical_objects}
Let $\alpha\in R\subset K(\sD)$ be a positive root. Let $\sigma\in\Stab^\dagger(\sD)$ be generic with respect to $\alpha$. Then, there exists a $\sigma$-stable object $E$ of class $\alpha$. The object $E$ is rigid if $\alpha$ is a real root, and it varies in a family if $\alpha$ is imaginary.
\end{thm}

We will make use of the following well-known property of K3-categories.

\begin{lemma}[{\cite[Prop. 2.9]{HMS08}}]\label{lem_spherical_has_spherical_JH}
Let $\sigma\in\Stab(\sD)$.
\begin{enumerate}[(i)]
    \item If $E\in\sD$ is spherical, then all of its $\sigma$-stable factors are spherical;
    \item  if $E\in\sD$ is semi-rigid, then all of its $\sigma$-stable factors are spherical, except for possibly one semi-rigid factor.
\end{enumerate}
\end{lemma}

Before moving forward, we recall a construction from \cite{BM14_MMP}.
Fix a primitive class $\bfv\in K(\sD)$, let $S$ be the set of objects of $\sD$ of class $\bfv$, and let $W=W^S_\bfw$ be a wall of the wall-and-chamber decomposition of $\Stab(\sD)$ associated to $\bfv$. Then we can associate to $W$ the rank 2 lattice $H_W\subset K(\sD)$:

\begin{equation}\label{eq_def_of_H_W}    H_W = \left\lbrace \bfw\in K(\sD) \mid \im \frac{Z(\bfv)}{Z(\bfw)}=0 \mbox{ for all } \sigma=(Z,\PP)\in W\right\rbrace.
\end{equation}

The rank of $H_W$ is at least 2 because it contains at least $\bfv$ and the linearly independent class $\bfw$ destabilizing at $W$. If it had rank bigger than 2, the definition \eqref{eq_def_of_H_W} would imply that $W$ has codimension higher than 1. 

For any $\sigma=(Z,\PP)\in W$, let $C_\sigma\subset H_W\otimes \R$ be the cone spanned by classes $\mathbf c$ satisfying
$$ \mathbf c^2\geq -2 \quad\mbox{and}\quad \im\frac{Z(\mathbf c)}{Z(\bfv)}>0. $$
We will refer to $C_\sigma$ as to the \emph{cone of $\sigma$-effective classes} in $H_W$.

\subsubsection{Wall-crossing for spherical classes}

%In this section we treat the case in which $\bfv$ is spherical.

\begin{lemma}\label{lem_lattice_H_W}
Let $\bfv$ be a primitive spherical class in $K(\sD)$, and $W$ be a wall for $\bfv$. Then $H_W$ is a primitive lattice of rank two generated by $\bfv$ and a spherical class $\bfw$. It is negative definite (with respect to the restriction of the Mukai pairing). Moreover, there are only three possibilities for the intersection form, and:
\begin{enumerate}[(i)]
    \item  if $(\bfv,\bfw)=0$, then $H_W$ contains no spherical classes except for $\pm\bfv$ and $\pm\bfw$;
    \item if $(\bfv,\bfw)=-1$, the only spherical classes in $H_W$ are $\pm\bfv$, $\pm\bfw$, and $\pm(\bfv-\bfw)$;
    \item if $(\bfv,\bfw)=1$, the only spherical classes in $H_W$ are $\pm\bfv$, $\pm\bfw$, and $\pm(\bfv+\bfw)$.
\end{enumerate}
\end{lemma}

\begin{proof}
We have that $\bfv\in H_W$ has $\bfv^2<0$ and $\bfw$ must be a spherical class by Lemma. \ref{lem_spherical_has_spherical_JH}. So both $\bfv$ and $\bfw$ project to non-zero vectors in $K(\sD)/\rad \chi$. The intersection matrix of $H_W$ can be computed on $K(\sD)/\rad \chi$, where the Mukai pairing coincides with the opposite of the Cartan intersection matrix, so it is negative definite.
%Let $\bar{\bfv}$ and $\bar{\bfw}$ be the images of $\bfv$, respectively $\bfw$ in $K(\sD)/\rad$. Then one has $(\bfv,\bfw)=(\bar{\bfv},\bar{\bfw})$. Since the Mukai pairing on $K(\sD)/\rad$ coincides with the opposite of the Cartan form, we see that $H_W$ must be negative definite. 

The signature of the form implies that the determinant of the intersection form be positive, which rules out all values of $(\bfv,\bfw)$ except for $0$ and $\pm 1$. The spherical classes are the integer solutions of 
$$ -2= (x\bfv+y\bfw)^2=-2x^2-2y^2 + 2(\bfv,\bfw)xy $$
in these three cases.
\end{proof}

Let $W$ be a wall for $\bfv$. Then, we denote by $\sigma_0$ a stability condition which only lies on the wall $W$, and consider a path in $\Stab(\sD)$ passing through $\sigma_0$ and connecting $\sigma^+$ and $\sigma^-$, two stability conditions lying in adjacent chambers.

\begin{lemma}\label{lem_wall_crossing_spherical}
For $W$ as above, suppose that there exists an indecomposable $\sigma_0$-semistable spherical object $E$ of class $\bfv$. Then there is a $\sigma^+$-stable spherical object $E^+$ of class $\bfv$. Likewise, there exist a $\sigma^-$-stable object $E^-$ of class $\bfv$.  
\end{lemma}

\begin{proof}
By Lemma \ref{lem_spherical_has_spherical_JH}, the Jordan-H\"older factors of $E$ are spherical objects. In other words, $\bfv$ can be written as a sum of spherical classes in $C_{\sigma_0}$. If $E$ is $\sigma_0$-stable, there is nothing to prove. Otherwise, Lemma \ref{lem_lattice_H_W} shows that, up to the sign of $\bfw$, $E$ has a Jordan-H\"older filtration $$B\to E\to A$$
where $B$, $A$ have class $\bfw$ and $\bfv -\bfw$, respectively. 
Observe that $\Ext^1(A,B)=\Ext^1(B,A)\neq 0$ since $E$ is indecomposable, and denote by $E'$ the non-trivial extension $$ A\to E'\to B. $$

If $\phi_{\sigma^+}(\bfv-\bfw) > \phi_{\sigma^+}(\bfw)$ set $E^+= E$. If $\phi_{\sigma^+}(\bfv-\bfw) < \phi_{\sigma^+}(\bfw)$, set $E^+= E'$.
In any case, $E^+$ satisfies the assumptions of \cite[Lemma 9.3]{BM14_MMP}, and hence is $\sigma^+$-stable.
\end{proof}

\subsubsection{Wall-crossing for radical classes}

\begin{lemma}\label{lem_lattice_H_W_radical}
Let $\bfv$ be a primitive radical class in $K(\sD)$, and $W$ be a wall for $\bfv$. Then $H_W$ contains a spherical class $\bfw$ and the intersection matrix of $H_W$ is $$\begin{pmatrix}0 & 0\\0 & -2\end{pmatrix}.$$ 
\end{lemma}

\begin{proof}
Another generator of $H_W$, $\bfw$, is either radical or semi-rigid by Lemma \ref{lem_spherical_has_spherical_JH}. If it is semirigid, $(\bfw,\bfw)=0$, so the intersection form is zero on $H_W$  and $H_W$ contains no spherical classes.
 Then every $\sigma_0$-semistable object $E$ of class $\bfv$ must be stable on $W$, because it can only have one Jordan-H\"older factor, so $W$ is not a wall.
The only other possibility is that $\bf w$ is spherical and the intersection form is as claimed.
\end{proof}

\begin{lemma}\label{lem_wall_crossing_radical}
For $W$ as above, suppose that there exists an indecomposable $\sigma_0$-semistable semi-rigid object $E$ of class $\bfv$. Then there is a $\sigma^+$-stable semi-rigid object $E^+$ of class $\bfv$. Likewise, there exist a $\sigma^-$-stable semi-rigid object $E^-$ of class $\bfv$.  
\end{lemma}

\begin{proof}
The proof is analogous to that of Lemma \ref{lem_wall_crossing_spherical}. If $E$ is $\sigma_0$-stable there is nothing to prove, otherwise it must have at least a spherical stable factor. Then one can write $\mathbf{v} = \mathbf{a} + \mathbf{b}$ with $\mathbf{a} \in C_{\sigma_0}$ spherical, and $\mathbf{b}\in C_{\sigma_0}$. By Lemma \ref{lem_lattice_H_W_radical}, the only spherical classes in $H$ are of the form $\pm\mathbf{w} + n\mathbf{v}$ with $n\in\Z$; then $\mathbf{b}$ has to be spherical as well, and there is only one integer $N$ such that $\mathbf{a} \coloneqq \mathbf{w} + N \mathbf{v}$ and $\mathbf{b}\coloneqq -\mathbf{w} + (1-N)\mathbf{v}$ are both $\sigma_0$-effective. Moreover, $\mathbf{a}$ and $\mathbf{b}$ cannot be expressed as the sum of other effective spherical classes. This implies that the Jordan-H\"older filtration of $E$ is
$$ \epsilon\colon \quad B\to E\to A. $$
Since $E$ is indecomposable, $(\epsilon)\neq 0$ in $\Ext^1(A,B)\simeq \Ext^1(B,A)$, and we can conclude as in Lemma \ref{lem_wall_crossing_spherical}. 
\end{proof}

\begin{proof}[{Proof of Theorem \ref{thm_enough_spherical_objects}}]
Suppose first that $\mathbf{v}$ is a spherical class. Proposition \ref{prop_CB_on_local_setting} shows that up to a sign there exists a $\tau_0$-semistable sheaf $E$ of class $\mathbf v$ which is spherical and indecomposable. Since $\Stab^\dagger(\sD)$ is connected and $\tau_0\in \Stab^\dagger(\sD)$, there is a path $\gamma$ of stability conditions in $\Stab^\dagger(\sD)$ connecting $\tau_0$ and $\sigma$. 

Observe that the objects $E^+$ produced in Lemma \ref{lem_wall_crossing_spherical} are in turn indecomposable, because they are stable with respect to some stability condition. Then, we can repeatedly apply Lemma \ref{lem_wall_crossing_spherical} and conclude.

A similar argument, where one uses Lemma \ref{lem_wall_crossing_radical} instead of Lemma \ref{lem_wall_crossing_spherical}, works for radical classes.
\end{proof}

\subsection{Proof of Proposition \ref{prop_circ=dagger}}\label{sec_proof_dagger=circ}

Now we prove that every stability condition in $\Stab^\dagger(\sD)$ (and hence in $\Stab_n^\dagger(\sD)$) satisfies 
\[ (*)\colon \quad  \im \frac{Z(b)}{Z(a)}>0.\]

It suffices to show that there does not exist a stability condition $\sigma_0=(Z_0,\sA_0)$ in $\Stab^\dagger(\sD)$ for which $\im\frac{Z(b)}{Z(a)}=0$. 

Suppose such $\sigma_0$ existed. Acting with $\C$, we may assume that $Z_0(a),Z_0(b)\in \R$. Assume moreover that $Z_0$ takes values in $\Q$. Then, choose $x,y\in\Z$ coprime such that 
\begin{equation}\label{eq_violate_support_property_0}
    xZ_0(a)+yZ_0(b)=0 
\end{equation}
and $\bfv\coloneqq xa+yb$ is a positive radical vector. Thus, $\bfv$ is a primitive radical vector with $Z_0(\bfv)=0$. This implies that there exists a neighborhood $V\subset \Stab^\dagger(\sD)$ of $\sigma_0$ such that no $\sigma\in V$ admits semistable objects of class $\bfv$, since semistability is a closed condition. But this contradicts Theorem \ref{thm_enough_spherical_objects}. 

If $Z_0$ takes values in $\R$, there may be no integer solutions to \eqref{eq_violate_support_property_0}, but for every $\epsilon >0$ there are integers $x,y$ such that 
$$ \abs{xZ_0(a)+yZ_0(b)}<\epsilon  $$
and $\bfv=xa+yb$ is a primitive radical vector. Choosing $\epsilon \ll 1$, the support property implies that there exists a neighborhood $V\subset \Stab^\dagger(\sD)$ of $\sigma_0$ such that no $\sigma\in V$ admits semistable objects of class $\bfv$, and we conclude in the same way. 

\subsection{Action of \texorpdfstring{$\C$}{C} and the orbit of normalized conditions}\label{ssec:C_Orbit}

Recall the $\C$-action on $\Stab(\sD)$ defined in Equation \eqref{eq_C_action} and denote by $\sK$ the orbit of $\Stab_n^\dagger(\sD)$. Here, we show that $\sK = \Stab^\dagger(\sD)$.  

It is straightforward to see $\sK \subseteq \Stab^\dagger(\sD)$, since $\sK$ is connected and intersects $\Stab^\dagger(\sD)$. 
%For $z=x+iy\in \C$, let $z \cdot (Z,\PP) = (Z',\PP')$, with
%\[ Z'(-) = e^{-z}Z(-), \qquad \PP'(\phi)=\PP\left(\phi+\frac{y}{\pi}\right). \] 
%This defines an action of $\C$ on $\Stab(\sD)$. 
To prove the other, fix $\tau\in\Stab^\dagger(\sD)$. By definition, there exists a path $\gamma\colon [0,1] \to \Stab^\dagger(\sD)$ such that $\gamma_0=\tau$ and $\gamma_1\in U$. We will use $\gamma$ to define $z_0\in \C$ and a modified path $\gamma'$, taking values in $\Stab^\dagger_n(\sD)$, such that $\gamma'_0=z_0\cdot\tau$, which shows $\tau\in \sK$. 

For every $t\in [0,1]$,  $\gamma_t=(Z_t,\PP_t)$ admits a semistable object $E_t$ of class $a$: this is true if $\gamma_t$ is generic by Theorem \ref{thm_enough_spherical_objects}, and hence for all $t$ since semistability is a closed condition.
Then define $\zeta_t\coloneqq Z_t(E_t)\in \C^*$ for all $t$. We can chose $E_t$ in a way that $\zeta\colon t \mapsto \zeta_t$ is continuous, hence a path in $\C^*$: since $E_0$ is $\gamma(0)$-semiststable, then it is semistable in an interval $[0,t_1]$ with $0\leq t_1\leq 1$, and hence we can pick $E_t=E_0$ for all $t\in [0,t_1]$. Since $\gamma(t_1)$ is at a wall for $a$, by Lemma \ref{lem_wall_crossing_radical} there exists $E_1$ which is $\gamma(t)$-semistable for $ t\in [t_1,t_2]$, with $t_1<t_2\leq 1$. Set $E_t=E_1$ for $t_1< t \leq t_2$. Since $Z_{t_1}(E_0)=Z_{t_1}(E_1)$, the function $\zeta$ is continuous at $t_1$. We can iterate this process since walls for $a$ are finite by Proposition \ref{prop_BM_finiteness_of_walls}. 

Since $\gamma_1\in \Stab^\dagger_n(\sD)$, we have $\zeta_1=1$, so the principal value $z\coloneqq \mathrm{Log}\,\zeta$ defines a continuous function $z\colon [0,1] \to \C$ such that $z_1=0$. We can finally define the path 
\begin{align*}
\gamma'\colon  [0,1] & \to \Stab^\dagger_n(\sD)\\
  t \, &\mapsto \, z_t\cdot \gamma(t).
\end{align*}

By construction, every stability condition $\gamma'(t)$ is normalized, and $\gamma'_1=\gamma_1\in U$. Then $\gamma'_0=z_0 \cdot \tau \in \Stab^\dagger_n(\sD)$, and $\tau_0\in \sK$.

If $\tau\in \Stab_n(\sD)$, the complex number $z_0$ has the form $z_0=i2\pi k$ for some $k\in \Z$, and acting with $z_0$ is the same as acting with $[2k]\in \Aut(\sD)$: in other words, the connected components of $\Stab_n(\sD)$ are even shifts of $\Stab_n^\dagger(\sD)$. Arguing as above one sees that $\Stab^\dagger(\sD)$ is a $\C^*$-bundle over $\Stab_n(\sD)$.
%....: If I have two normalized stability conditions within the same connected component, do I know they are in the same $\C$-orbit? Why? In other words, how can I produce a different heart (not $\sA[n]$) for which $Z'\in D$ works? 

%----------------------------------------------------------------------------------------------------------------------------------------

\section{Stability conditions on \texorpdfstring{$\sD$}{mathcal D}}\label{sec_stabCondOnD}

In this section we study the action of $\Br(\sD)$ on $\Stab(\sD)$ and show that it preserves $\Stab_n^\dagger(\sD)$. Then, we describe the image of $\Stab_n^\dagger(\sD)$ in $\Hom(K(\sD),\C)$ and show $\pi(\Stab_n^\dagger(\sD))=\mathsf{X}_{\mathrm{reg}}$ (Prop. \ref{prop_image_of_StabDagger}). Finally, we prove our main results in Section \ref{ssec_MainResults}.

\subsection{Group actions and the image of the central charge map}

The group of autoequivalences of $\sD$ acts on $\Stab(\sD)$ as in Equation \eqref{eq_Aut_Action}.
% for $\Phi\in \Aut(\sD)$ and $\sigma=(Z,\PP)\in \Stab(\sD)$, define $\Phi\cdot (Z,\PP)=(Z',\PP')$ as the stability condition with
%$$ Z'(E)\coloneqq Z(\Phi^{-1} (E)) \;\text{ and }\; \PP'(\phi)\coloneqq \Phi(\PP(\phi)). $$ 
The following discussion shows that the autoequivalences in $\Br(\sD)$ preserve $\Stab_n^\dagger(\sD)$. It follows that the central charge map is equivariant with respect to the actions of $\Br(\sD)$ and $W$ on $\Stab_n^\dagger(\sD)$ and $\Hom(F,\C)$ respectively.

Recall from Section \ref{sec_fund-dom-and-boundary} that the boundary of $D$ (defined as a fundamental domain of $W$ in $\Hom(F,\C)$) is contained in the union of $Y_{u,\pm}$ walls $W_{v,\pm}$ as $u,v$ vary in the vertices of $\abs{\Gamma_f}$ and $\abs{\Gamma_a}$ respectively. Denote by $\tilde{Y}_{u,\pm}$, $\tilde{W}_{v,\pm}$ the inverse images of $Y_{u,\pm}$, $W_{v,\pm}$ to $\overline{U}$ (we use Prop. \ref{prop_FundamentalRegionU} here). 

\begin{lemma}\label{lem_Arend_concern}
Let $\sigma=(Z,\sA)$ be a point in the boundary of $U$. Then $\sigma$ lies in the union of $\tilde{W}_{v,\pm}$, $\tilde{Y}_{u,\pm}$.
\end{lemma}

\begin{proof}
This follows from the description of the boundary of $D$ in Sec. \ref{sec_fund-dom-and-boundary}: the only other possibility is that $\im Z(b)=0$, but this is excluded by Proposition \ref{prop_circ=dagger}.
%Since $\tau = \sum_{i=0}^l a_i\alpha_i$ is a linear combination of the $\alpha_i$ with positive coefficients, this implies $\im Z(\alpha_i)=0$ for $i=0,...,l$, so $Z(K(\sD))\in \R$....
\end{proof}

Recall the notation of Equation \eqref{eq_notation_Si}, and let $v\in \abs{\Gamma}$:

 \begin{lemma}
Let $\sigma=(Z,\sA)$ be a point in the boundary of $U$ contained in a unique wall among the $\tilde{W}_{v,\pm}$'s.
Then there is an element $T\in\Br(\sD)$ such that $T\cdot\sigma$ also lies in the boundary of $U$. More precisely, we may pick $T=\Phi_{S_v}$  if $\sigma \in \tilde{W}_{v,+}$, and $T=\Phi^{-1}_{S_v}$ if $\sigma \in \tilde{W}_{v,-}$.
\end{lemma}

\begin{proof}
Suppose $\sigma\in \tilde{W}_{v,-}$. Set $S\coloneqq S_v$. Let $V$ be a small neighborhood of $\sigma\in\Stab(\sD)$, and consider the open subset
$$ V^+=\set{\sigma'=(Z',\sA') \in V \st \im Z(S)<0}. $$
Arguing as in \cite[Lemma 3.5]{Bri09_kleinian}, we claim that we can choose $V$ small enough so that  $\Phi_S^{-1}(V^+)\subset U$, hence $\Phi^{-1}_S\sigma$ lies in the closure of $U$. 
Thus, we need to show that for sufficiently small $V$ the heart of all $\sigma'\in V^+$ is equal to $\Phi_S(\mathcal{A}_R)\subset \sD$. By Lemma \ref{lem_HeartsContainEachOther}, it suffices to show that $\Phi_S(M)$ lies in the heart of any $\sigma'\in V^+$, for all the objects $M$ listed in Prop. \ref{prop_classification-objects-A}. 

We verify this on a case by case basis: assume first that $S=t_i^j$, $j\neq 0$. Then:

\noindent\textit{Case 1.} Suppose $L$ is a line bundle on $X$. Then $L$ is locally of the form $\sO((k/a_{i})p_i)$ for some $k\in\set{0,...,a_i}$, and one computes 
    $$\Hom^\bullet (t_i^j,L)=\begin{cases}
    \C[-1] \mbox{ if }k=j\\ 
    \C[-2] \mbox{ if }k+i=j\\
    0 \mbox{ otherwise.}
    \end{cases}$$
    
    If $\Hom^1(t_i^j,L)\neq 0$, then there is a non-split short exact sequence in $\mathcal{A}_R$
    $$ L\to \Phi_SL \to t_i^j. $$
    It follows that $\Phi_S L$ lies in the heart of $\sigma$ and its semistable factors have phases in $(0,1)$. Choosing $V$ small enough ensures that this is the case for all $\sigma'\in V^+$ too. 
    
    If $\Hom^2(t_i^j,L)\neq 0$ then $\Phi_SL$ fits in a triangle
    $$ L\to \Phi_SL \to t_i^j[-1], $$
    which implies that $\Phi_SL$ lies in $\sA'$, because so do $L$ and $t_i^j[-1]$.
    
    If $\Hom^\bullet(t_i^j,L)= 0$ then $\Phi_SL=L$ and the same argument applies.
    
    \noindent\textit{Case 2.} The same argument applies to $\Phi_{t_i^j}(\sO_q)=\sO_q$ for all $q\neq p_1$, and to all sheaves supported away from $p_i$;

    \noindent\textit{Case 3.} The only possibilities for $\Phi_St_i^k$, $k\neq j,0$ are that $\Hom^\bullet(t_i^j,t_i^k)=0$ or $\Hom^\bullet(t_i^j,t_i^k)=\C[-1]$. Both are analogous to the case of a line bundle above. Consider $\Phi_S(S)=S[-1]$. Since $S$ is $\sigma$-stable of phase 1, we may assume by shrinking $V$ that $S$ is $\sigma'$-stable with phase at most 2. Moreover, $S$ must have phase bigger than 1 in $\sigma'$, so $S[-1]$ lies in the heart of $\sigma'$. Similarly, one sees that $\Phi_St_i^0[-1]\in \sA'$.
    
    \noindent\textit{Case 4.} If $M$ is a cluster supported at $p_i$, then $M$ has a non-split composition series with factors the $t_i^j$ for $j=0,...,a_i-1$, where $t_i^0$ is the last factor. Then, $\Phi_SM$ has a non-split composition series with all factors in $\sA'$ but the last one in $\sA'[1]$, and $Z'(\Phi_SM)=-Z'(a)=-1$, so $\Phi_S(M)\in \sA'$.
    
    \noindent\textit{Case 5.} It remains to show the claim for $N[-1]$ where $N$ is the proper quotient of a cluster $M$, with kernel $K$. Write the triangle
    \begin{equation}
        \label{eq_triangleM[-1]_N[-1]_K}
     M[-1] \to N[-1] \to K
    \end{equation}
    and apply $\Phi_S$. By the discussion above, $\Phi_S(K)\in\sA'$ since $K$ is obtained by repeated extensions of $t_i^j$'s with $j>0$, and $\Phi_S(M)[-1]$ is stable of phase 0. Then $\Phi_S(N)[-1]\in\sA'$, because the triangle \eqref{eq_triangleM[-1]_N[-1]_K} does not split.

Similar computations show that $\Phi_S(M)\in \sA'$ for all $M\in\mathcal{A}_R$ and $S=\sO_{X}$. A symmetric argument settles the case $\sigma\in \tilde{W}_{v,+}$. 
\end{proof}

\begin{lemma}
Let $\sigma=(Z,\sA)$ be a point in the boundary of $U$ contained in a unique wall among the $\tilde{Y}_{u,\pm}$.  
Then there is an element $T\in\Br(\sD)$ such that $T\sigma$ also lies in the boundary of $U$. More precisely, we may pick $T=\rho_u$  if $\sigma \in \tilde{Y}_{u,+}$, and $T=\rho_{u}^{-1}$ if $\sigma \in \tilde{Y}_{u,-}$.
\end{lemma}

\begin{proof}
If $\sigma\in \tilde{Y}_{u,+}$, observe that we can choose a small neighborhood $V$ of $\sigma$ in $\Stab(\sD)$ so that every $\tau\in V$ has heart $\mathcal{A}_R$. Consider the open subset
$$ V'=\set{\tau=(Z',\mathcal{A}_R)\in V \st \tau \notin \bar U} $$
For $\tau\in V'$, we then have that $\rho_u^{-1}Z'= \rho_u^{-1}\re Z' +i\im Z'$ belongs to $D$. Then, it is enough to show $\rho_u(\mathcal{A}_R)=\mathcal{A}_R$ to conclude $\rho_u\tau\in U$, so that $\rho_u\sigma$ lies in the closure of $U$. 

Using Prop. \ref{prop_classification-objects-A}, one sees that $\PP_\sigma(1)$ only contains objects whose class is a multiple of $a$. Since $\rho_u$ preserves the imaginary part of $Z'$ and fixes the class $a$, we have $\PP_\tau(1)=\PP_\sigma(1)$. Then, the only possibility is that for $u\in \abs{\Gamma_f}$ one has $\rho_u(\mathcal{A}_R)=\mathcal{A}_R[2n]$, for some integer $n$. We prove that $n$ must be $0$. 
One readily checks
$$ \rho_{(0,1)}(\sO_X(1)) =\Phi_{\sO_X}\Phi_{\sO_X(1)}(\sO_X(1)) \simeq \Phi_{\sO_X}(\sO_X(1)[-1]) = \sO_X(-1). %\Phi_{\Phi_{\sO_X}(\sO_{X}(1))}\Phi_{\sO_X}\sO_X=\sO_X(-2)
%\simeq \Phi_{\sO_X(-1)}\sO_X[-1] =\sO_X(-2). 
$$
using Lemma \ref{lem_properties_of_twists}. This implies that $\rho_0(\mathcal{A}_R)=\mathcal{A}_R$. Now one has
\begin{equation}
    \begin{split}
       \rho_{(i,1)}(\sO_X(-1)) & = \Phi_{(t_i^1)}\rho_{(0,1)}\Phi_{(t_i^1)}\rho_{(0,1)}^{-1}(\sO_X(-1))\\
                            & \simeq \Phi_{(t_i^1)}\rho_0\Phi_{(t_i^1)}(\sO_X(1))\\
                            & \simeq \Phi_{(t_i^1)}\Phi_{\sO_X}\Phi_{\sO_X(1)}\Phi_{(t_i^1)}(\sO_X(1))\\
                            & \simeq \Phi_{(t_i^1)}\Phi_{\sO_X}(t_i^1)\\
                            & \simeq \sO_X                      , 
    \end{split}
\end{equation}
by repeatedly applying Lemma \ref{lem_properties_of_twists}. For $\rho_{(i,j)}$, $j>1$, we claim $\rho_{(i,j)}(\sO_X)\simeq \sO_X$. This is a consequence of the fact that $\sO_X(d)$ is orthogonal to $t_i^j$ for $d=0,-1$, all $i$ and all $j>1$. Indeed, one computes
\begin{equation}
    \begin{split}
         \rho_{(i,2)}(\sO_X) & = \Phi_{(t_i^2)}\rho_{(i,1)}\Phi_{(t_i^2)}\rho_{(i,1)}^{-1}(\sO_X)\\
                            & \simeq \Phi_{(t_i^2)}\rho_{(i,1)}\Phi_{(t_i^2)}(\sO_X(-1))\\
                            & \simeq \Phi_{(t_i^2)}\rho_{(i,1)}(\sO_X(-1))\\
                            & \simeq \Phi_{(t_i^1)}(\sO_X)\\
                            & \simeq \sO_X,
    \end{split}
\end{equation}
and proves the same claim for $j>2$ inductively. This concludes the proof in the case $\sigma\in \tilde{Y}_{i,+}$. The case $\sigma\in \tilde{Y}_{i,-}$ is similar.
\end{proof}

\begin{comment}
Let $\sT$ denote the subcategory of $\Coh_0(\Tot(\omega_X))$ generated by the objects $\sO_{p_i}$, $i=1,2,3$, and let $\sF=\sT^\perp$ be its right orthogonal. 

\begin{lemma}
The subcategories $(\sT,\sF)$ are a torsion pair $\Coh_0(X_r)$.
\end{lemma}

\begin{proof}
Fix $E\in \Coh_0(X_r)$. Let $a_i=\dim \Hom(\sO_{p_i},E)$. Then we have a short exact sequence  
$$  0\to \sO_{p_1}^{a_1} \oplus \sO_{p_2}^{a_2} \oplus \sO_{p_3}^{a_3} \to E \to E' \to 0. $$
Applying $\Hom(\sO_{p_i},-)$ one sees $\Hom(\sO_{p_i},E')=0$ for $i=1,2,3$, so $E'\in \sF$.
\end{proof}

\end{comment}

%Let $\pi$ be the restriction of the central charge map to $\Stab^\dagger(\sD)$, and define $\Stab^\dagger(\sD)^N$ to be 
%$$ \Stab^\dagger(\sD)^N \coloneqq \pi^{-1}\mathbb E  $$
\begin{comment}
As a consequence of Lemma \ref{lem_Arend_concern}, we have
\begin{corollary}
$\Stab^\dagger(\sD)^N$ is a slice of $\Stab^\dagger(\sD)$ with respect to the $\C$-action on $\Stab(\sD)$. In other words, for every $\sigma \in \Stab^\dagger(\sD)$ there exists $t\in\C$ such that $t\sigma\in \Stab^\dagger(\sD)^N$.\AAA{this true? We need to worry about this only later, when we prove that the image is exactly Xreg..}
\end{corollary}
\end{comment}

\begin{proposition}\label{prop_fund_domain_br}
For any $\sigma \in \Stab_n^\dagger(\sD)$, there is an autoequivalence $\Phi\in\Br(\sD)$ such that $\Phi\cdot\sigma\in U$.
\end{proposition}

\begin{proof}
Same as the proof of Prop. 4.13 in \cite{Ike14}.
\end{proof}

Let $\pi^{-1}(\mathsf{X}_{\mathrm{reg}})^\dagger$ be the connected component of $\pi^{-1}(\mathsf{X}_{\mathrm{reg}})$ containing $U$. Since it is a subset of $\Stab_n^\dagger(\sD)$ we have:

\begin{corollary}\label{cor_fund_domain_br_C}
For any $\sigma \in \pi^{-1}(\mathsf{X}_{\mathrm{reg}})^\dagger$, there is an autoequivalence $\Phi\in\Br(\sD)$ %and $k\in\C$ s
such that $\Phi\cdot \sigma \in U$.
\end{corollary}

%\begin{proof}
%See \cite[Cor. 4.14]{Ike14}.
%\end{proof}

\begin{lemma}\label{lem_image_contains_Xreg}
The image of $\pi\colon \Stab_n^\dagger(\sD)\to \Hom(F,\C)$ contains $\mathsf{X}_{\mathrm{reg}}$.
\end{lemma}

\begin{proof}
$\Stab_n^\dagger(\sD)$ contains the orbit of $U$ under $\Br(\sD)$. Since the action of $\Br(\sD)$ lifts that of $W$ on $\Hom(F,\C)$, the orbit of $U$ under the action of $\Br(\sD)$ is mapped to $\mathsf{X}_{\mathrm{reg}}\subset \Hom(F,\C)$.
\end{proof}

The next goal of our discussion is to prove the following:
\begin{proposition}\label{prop_image_of_StabDagger}
The projection $\pi$ maps $\Stab_n^\dagger(\sD)$ onto $\mathsf{X}_{\mathrm{reg}}$, so that $\pi(\Stab_n^\dagger(\sD))=\mathsf{X}_{\mathrm{reg}}$. % by Lemma \ref{lem_image_contains_Xreg}.
\end{proposition}

\begin{proof}
By Lemma \ref{lem_image_contains_Xreg}, it is sufficient to show that $\pi(\Stab_n^\dagger(\sD))\subseteq \mathsf{X}_{\mathrm{reg}}$, or, equivalently, that $\Stab_n^\dagger(\sD)\subseteq \pi^{-1}(\mathsf{X}_{\mathrm{reg}})^\dagger$. To show this, it is enough to check that $\Stab_n^\dagger(\sD)$ contains no boundary points of $\pi^{-1}(\mathsf{X}_{\mathrm{reg}})^\dagger$. Any such boundary point $\sigma=(Z,\PP)$ is projected to $Z\in\partial \mathsf{X}_{\mathrm{reg}}$. From the definition of $\mathsf{X}_{\mathrm{reg}}$ in Prop. \ref{prop_Saito_action_of_W}, either $Z$ vanishes on a ray in $\R_{>0}(R)$, or $\im Z(b)=0$.
%By Lemma \ref{lem_def of partialXreg}, this implies that there is a ray in $\R_{>0}(R)$ such that $Z(S)=0$, or $\im\frac{Z(b)}{Z(a)}=0$. 

In the latter case, Proposition \ref{prop_circ=dagger} ensures that $\sigma \notin \Stab_n^\dagger(\sD)$. Then, suppose $\alpha$ is a positive root such that $Z(\alpha)=0$. If $\sigma\in\overline{\Stab_n^\dagger(\sD)}$, by proposition \ref{prop_fund_domain_br} there is an element $\Phi\in\Br(\sD)$, such that $\Phi\cdot \sigma = (Z',\PP') \in \overline{U}$, and $[\Phi] \alpha=\beta\in \Pi$. Then we have $Z'(\beta)=0$. However, by Lemma \ref{lem_Pi_are stable in U}, for all $\beta\in\Pi$ there are objects of class $\beta$ which are semistable for all stability conditions in $U$, hence $\Phi\cdot\sigma$ violates the support property, and therefore $\sigma\notin \Stab_n^\dagger(\sD)$.
\end{proof}

\begin{comment}
\begin{lemma}\label{lem_def of partialXreg}
Let $Z\in\partial \mathsf{X}_{\mathrm{reg}}$. Then either $\im\frac{Z(b)}{Z(a)}=0 $ or there is a positive root $\alpha$ such that $Z(\alpha)=0$.
%$S\in \R_{>0}\Delta_+^{re}(R)$ such that $Z(S)=0$
\end{lemma}
\begin{proof}
This follows from the definition of $\mathsf{X}_{\mathrm{reg}}$ in Section \ref{sec_the_regular_set}.
\end{proof}
\end{comment}

\begin{proposition}
The action of $\Br(\sD)$ on $\Stab_n^\dagger(\sD)$ is free and properly discontinuous. 
\end{proposition}

\begin{proof}
%This is clear for the action of $\Z[2]$, so it is enough to check it for $\Br(\sD)$.
First, we check that the action of $\Br(\sD)$ is free. By Cor. \ref{cor_fund_domain_br_C}, it is enough to show this for $\sigma\in U$. Assume then that $\sigma=\Phi\sigma$ for some $\Phi\in \Br(\sD)$ and $\sigma\in U$. We have $Z(\Phi(-))=Z(-)$, hence $[\Phi]=\id$ on $K(\sD)$. So $[\Phi(S_m)]=[S_m]$ for all $m$. Up to isomorphism, $S_m$ is the only object in $\mathcal{A}_R$ in its class (this is readily observed translating $\mathcal{A}_R$ to $\Psi^{-1}(\mathcal{A}_R)$), hence $\Phi(S_m)\simeq S_m$ for all $m$. Then $\Phi\simeq \id$ in $\Br(\sD)$ by Lemma \ref{lem_Phi(S_m)=S_m}.

To show that the action of $\Br(\sD)$ is properly discontinuous, it is enough to exhibit, for every non-trivial $\Phi\in\Br(\sD)$ and every $\sigma\in U$, a neighborhood $V$ of $\sigma$ such that $\Phi (V)\cap V=\emptyset$. If $[\Phi]\neq\id$, the existence of $V$ follows from Prop. \ref{prop_Saito_action_of_W}. If $[\Phi]=\id$, then it is a consequence of Lemma \ref{lem_BriCloseAreEqual}.
\end{proof}

% Added lemma thanks to Michael Wemyss

\begin{lemma}\label{lem_Phi(S_m)=S_m}
Suppose $\Phi\in\Br(\sD)$ satisfies $\Phi(S)\simeq S$ for all $S\in\Pi$. Then $\Phi\simeq\id$. 
\end{lemma}

\begin{proof}
We consider $\Phi$ as an element of $\Aut(D^b(\Tot(\omega_X)))$, and we study the equivalent problem of showing that
$$ \Phi'\coloneqq \Psi^{-1} \circ \Phi \circ \Psi $$
is the identity on $\Aut(D^b(Y'))$, where $Y'$ denotes the crepant resolution of $\Tot(\omega_X)$, under the assumption that elements of $\Psi^{-1}\Pi$ are fixed (recall the notation of Section \ref{sec_tri_cat_loc_ell_quot}).

First, observe that for $p\in Y'\setminus X'$ we have $\Phi(\sO_p)\simeq \sO_p$ because all $S\in \Pi$ are supported on $X$ and hence orthogonal to $\sO_p$. If $p\in X\subset X'$, applying $\Phi$ to the short exact sequence
\[ 0\to i_*\sO_X(-1) \xrightarrow{f} i_*\sO_X \to \sO_p \to 0  \]
one obtains a non zero map $\Phi(f)$ of pure one-dimensional sheaves, fitting in a triangle
\[ i_*\sO_X(-1) \xrightarrow{\Phi(f)} i_*\sO_X \to \Phi(\sO_p).  \]
This implies that $H^{-1}\Phi(\sO_p)=0$ and $\Phi(\sO_p)$ is a skyscraper supported at a point of $X$. 

Now let $\set{p}= X\cap C_{i,1}$. Then the skyscraper supported at $p$ must be fixed by $\Phi'$, because it admits a restriction map $\sO_{C_{i,1}}(-1)\to \sO_{p}$ and $\Phi'$ fixes $\sO_{C_{i,1}}(-1)=\Psi^{-1}t_i^1$. Let $M_p$ denote the cluster corresponding to $p$. Then $\Phi$ fixes $M_p$ because $\Phi'$ fixes $\sO_p$. Moreover, $M_p$ has a unique composition series by the $t_i^j$, which are all fixed by $\Phi$ except possibly $t_i^0$. Then $\Phi$ must also fix $t_i^0$ for $i=1,...,r$. 

Then, since every cluster has a composition series with factors the simple sheaves $t_i^j$ and $\Phi$ fixes the $t_i^j$ for all $j=0,...,a_i-1$, it must also send any cluster to a cluster. In other words, $\Phi'$ sends skyscraper sheaves of points on any exceptional curve $C_i$ to skyscraper sheaves. 

Once can then apply \cite[Cor. 5.23]{Huy06}, which implies that there exists an automorphism $\phi$ of $Y'$ such that  $\Phi'(\sO_t)\simeq \sO_{\phi(t)}$ and $\Phi'\simeq (- \otimes \sL)\circ \phi_*$ for some line bundle $\sL$ on $Y'$. The automorphism $\phi$ is the identity, because it is the identity on the dense open complement of $X'$. The Picard group of $\Tot(\omega_X)$ is isomorphic to $\Pic(X)\bigoplus(\oplus \Z\set{C_{i,j}})$ hence the only line bundle fixing the $\Psi^{-1}(S)$ with $S\in\Pi$ is the trivial one. Then, $\Phi'\simeq \id$ as we wished to prove.
\end{proof}

\subsection{Proof of main results} \label{ssec_MainResults}

Denote  by $\bar\pi$ the composition of the maps $\Stab_n^\dagger(\sD)\xrightarrow{\pi} \mathsf{X}_{\mathrm{reg}}\to \mathsf{X}_{\mathrm{reg}}/W$. Then we have:
 
\begin{thm}\label{thm_mainThm}
The map 
$$\bar\pi\colon\Stab_n^\dagger(\sD) \to \mathsf{X}_{\mathrm{reg}}/W$$
is a covering map, and the group $\Br(\sD)$ acts as group of deck transformations.
\end{thm}

\begin{proof}
We only need to show that the quotient of $\Stab_n^\dagger(\sD)$ by $\Br(\sD)$ coincides with $\mathsf{X}_{\mathrm{reg}}/W$. Equivalently, for every pair of stability conditions $\sigma_1$, $\sigma_2$ satisfying $\bar{\pi}(\sigma_1)=\bar{\pi}(\sigma_2)$, we need to exhibit an element $\Phi\in\Br(\sD)$ such that $\sigma_1=\Phi \cdot \sigma_2$.

By Corollary \ref{cor_fund_domain_br_C}, it is enough to show this when $\sigma_1 \in U$. Moreover, there exists $\Phi\in\Br(\sD)$ such that $\sigma_2'\coloneqq \Phi \cdot \sigma_2$ lies in $U$. Then we have 
$$ \pi(\sigma_2') = [\Phi]\cdot \pi(\sigma_2)= [\Phi]\cdot \pi(\sigma_1)$$
in $D$. Since $U$ and $D$ are homeomorphic, this implies that $[\Phi]=\id$ and $\sigma_2'=\sigma_1$. 
\end{proof}

Let $\Aut^\dagger(\sD)\subset \Aut(\sD)$ be the subgroup of autoequivalences which preserve the component $\Stab_n^\dagger(\sD)$. Write $\Aut^\dagger_*(\sD)$ for the quotient of $\Aut^\dagger(\sD)$ by the subgroup of autoequivalences which act trivially on $\Stab_n^\dagger(\sD)$. 

\begin{corollary}\label{cor_autoequivalences}
There is an isomorphism
$$\Aut^\dagger_*(\sD) \simeq  \Br(\sD) \rtimes \Aut(\Gamma), $$
Where $Aut(\Gamma)$ acts on $\Br(\sD)$ by permuting the generators. 
\end{corollary}

\begin{proof}
The argument is the same as \cite[Cor. 1.4]{Bri09_kleinian}. Observe that unlike in \cite{Bri09_kleinian}, the shift autoequivalence does not belong to $\Aut^\dagger(\sD)$, since it maps  $\Stab_n^\dagger(\sD)$ to a different connected component in $\Stab_n(\sD)$ if it's an even shift or outside $\Stab_n(\sD)$ if it's odd.
\end{proof}

\bibliographystyle{amsplain}
\bibliography{./bibliography}

\begin{comment}
%LMS command:
\affiliationone{% in this example, two authors share an institution
   Franco Rota\\
   Department of Mathematics, Rutgers University\\
      Hill Center, Piscataway, NJ 08854\\
   United States
   \email{rota@math.rutgers.edu}}
% Important: Do not put any empty line here.
\end{comment}

\end{document}